\documentclass[12pt, leqno]{amsart}
\usepackage{amsmath}
\usepackage{amssymb}
\usepackage{amsthm}
\usepackage{enumerate}
\usepackage[mathscr]{eucal}
\def\mymathhyphen{{\hbox{-}}}
\theoremstyle{plain}
\usepackage{tikz}
\newtheorem{theorem}{Theorem}[section]
\newtheorem{lemma}[theorem]{Lemma}
\newtheorem{prop}[theorem]{Proposition}
\theoremstyle{definition}
\newtheorem{definition}[theorem]{Definition}
\newtheorem{remark}[theorem]{Remark}

\newtheorem{example}[theorem]{Example}

\newtheorem{cor}[theorem]{Corollary}
\theoremstyle{remark}
\begin{document}
	\title[On some subspaces of vector-valued continuous function space]{On some subspaces of vector-valued continuous function space, from the perspective of Best coapproximation}
	\author[ Souvik Ghosh, Kallol Paul, Debmalya Sain and Shamim Sohel]{ Souvik Ghosh, Kallol Paul, Debmalya Sain and Shamim Sohel}

	\newcommand{\acr}{\newline\indent}

	\address[Ghosh]{Department of Mathematics, Jadavpur University, Kolkata 700032, West Bengal, INDIA}
	\email{sghosh0019@gmail.com; souvikg.math.rs@jadavpuruniversity.in}
	
	\address[Paul]{Vice-Chancellor, Kalyani University, West Bengal and Professor of Mathematics (on lien)\\ Jadavpur University\\ Kolkata 700032\\ West Bengal\\ INDIA}
	\email{kalloldada@gmail.com}
	
	\address[Sain]{Department of Mathematics\\ Indian Institute of Information Technology\\ Raichur 584135\\ Karnataka \\INDIA\\ }
	\email{saindebmalya@gmail.com}
	
		\address[Sohel]{Department of Mathematics, Jadavpur University, Kolkata 700032, West Bengal, INDIA}
	\email{shamimsohel11@gmail.com}

	%\address[Mal]{Department of Mathematics\\ Jadavpur University\\ Kolkata 700032\\ West Bengal\\ INDIA}
	%\email{arpitamalju@gmail.com}
	
	%\address[Mandal]{Department of Mathematics\\ Jadavpur University\\ Kolkata 700032\\ West Bengal\\ INDIA}
	%\email{kalidas.mandal14@gmail.com}

	\thanks{ Souvik Ghosh and Shamim Sohel  would like to thank CSIR, Govt. of India, for the financial support in the form of	Senior Research Fellowship under the mentorship of Prof. Kallol Paul.} 
	
	%\subjclass[2010]{Primary 46B28, 47L05 Secondary 47L25}

	\begin{abstract}

		This article explores anti-coproximinal and strongly anti-coproximinal subspaces in the  spaces of vector-valued continuous functions and operator spaces. We provide a complete characterization of strongly anti-coproximinal subspaces in $ C_0(K, \mathbb{X}) $, under the assumption that the unit ball of  $ \mathbb{X}^* $ is the  closed convex hull of its weak*-strongly exposed points.  Additionally, the work includes a  stability analysis of  anti-coproximinal and strongly anti-coproximinal subspaces  of $ \mathbb{L}(\mathbb{X}, \mathbb{Y}) $ and the  space $ \mathbb{Y} $. Beyond these, we present a general characterization of (strong) anti-coproximinal subspaces in the broader context of Banach spaces.

	\end{abstract}
	
	\maketitle
	\section{Introduction.} 
	
	The notion of best coapproximation was introduced by Francetti and Furi in \cite{FF}. Since then, it has been studied in connection with various aspects of Banach space geometry, including the classical problem of one-complemented subspaces (see, for example, \cite{ Bruck, LT, PS, Rao, SSGP, SSGP2}). In our recent work \cite{SSGP3}, we introduced and investigated two special types of subspaces from the perspective of best coapproximation, namely the anti-coproximinal and strongly anti-coproximinal subspaces. Subsequently, in \cite{SGSP}, we explored these subspaces within the space 
	$C(K)$, the Banach space of all scalar-valued continuous functions defined on a compact Hausdorff space $K$.
	In the present article, we extend this line of inquiry by characterizing these special subspaces in the space of vector-valued continuous functions defined on a compact Hausdorff space $K$. Before presenting the main results, we first establish the necessary notation and terminology that will be used throughout the paper.
	\\
	
	Letters $\mathbb{X}, \mathbb{Y}$ denote Banach spaces over the field $\mathbb{K},$ either real or complex and $\mathbb{X}^*$ is the dual of $\mathbb{X}$. 	For $\epsilon >0,$  let $\mathcal{D}(\epsilon)= \{z \in \mathbb{K}: |z| \leq \epsilon \}$  and let $\mathbb{T}$ denote the unit circle in the complex plane.  For a complex number $z,$ we denote $\Re[z]$  for the real part  of $z.$	 $B_\mathbb{X}$ and $S_{\mathbb{X}}$ stand for unit ball and unit sphere of $\mathbb{X},$ respectively.  Let $\mathbb{L}(\mathbb{X}, \mathbb{Y})$ $(\mathbb{K}(\mathbb{X}, \mathbb{Y}))$ be the space of all bounded (compact) linear operators between $\mathbb{X}$ and $\mathbb{Y}.$ The space of all finite-rank operators is denoted by $\mathcal{F}(\mathbb{X}, \mathbb{Y}).$   The annihilator  of  a subset $S$ of $\mathbb{X}$   denoted by $S^\perp,$ is defined as  $S^\perp := \{x^* \in \mathbb{X}^* : x^* (x)=0, \, \text{for all}\,\, x \in S\}.$ Similarly, if $M$ is a subset of $\mathbb{X}^*$ then $^\perp M$ is defined as: $^\perp M := \{x \in \mathbb{X} :  x^* (x)=0, \, \text{for all}\, x^* \in M\}.$ 
	For a nonzero $x \in \mathbb{X},$  $x^* \in S_{\mathbb{X}^*}$ is said to be a supporting functional at $x$ if $x^*(x)=\|x\|.$ The set of all supporting functionals at $x$ is denoted by $J(x),$ i.e., $J(x)=\{x^* \in S_{\mathbb{X}^*}: x^*(x)=\|x\|\}.$  The set valued map $J: \mathbb X \to \mathbb{X}^*,$ where $J(x)=\{x^*\in S_{X^*} : x^*(x)=\|x\|\}$ is said to be the duality map. A nonzero  element $x \in \mathbb{X} $ is said to be smooth if $J(x)$ is  singleton.
	The convex hull of a set $S$ is denoted as $co(S).$  The set of all extreme points of $C$ is denoted by $Ext(C).$ A finite-dimensional real Banach space is said to be polyhedral  if $Ext(B_{\mathbb{X}})$ is finite. A convex set $F \subset S_{\mathbb{X}}$ is said to be a face of $B_{\mathbb{X}} $ if for any $ y, z \in B_{\mathbb{X}},$ $\frac{1}{2}(y+z) \in F$ implies that $y, z \in F.$  $F$ is called a maximal face if for any face    $F'$ of $B_\mathbb{X},$ $F \subset F'$ implies $F=F'.$  The notation $int(F)$ denotes the relative interior of a face $F$ endowed with the subspace topology of $F.$  
	The space  $ \mathbb X$ is said to be strictly convex if $Ext(B_{\mathbb{X}})= S_{\mathbb{X}}.$ An element $x \in S_\mathbb{X}$ is said to be rotund if for some  $y \in B_\mathbb{X},$ $\|\frac{x+y}{2}\|=2$ implies $x=y.$ In a strictly convex space every element of the unit sphere is  rotund.  
	An element $x \in S_{\mathbb{X}}$ is said to be locally uniformly rotund, in short,  LUR \cite{BHL, M} if given any $ \{x_n\}_{n \in \mathbb{N}} \subset B_\mathbb{X},$ $$\lim_{n\to \infty}\frac{\|x_n+x\|}{2} = 1$$ implies that $\lim_{n\to \infty} x_n = x.$
	An element $x \in S_{\mathbb{X}}$ is said to be almost locally uniformly rotund or ALUR (w-ALUR) point if given any $\{x_m^*\}_{m \in \mathbb{N}} \subset B_{\mathbb{X}^*}$ and any $\{x_n\}_{n \in \mathbb{N}} \subset B_\mathbb{X} $, $$\lim_{m\to \infty}\lim_{n\to\infty} x_m^* \bigg(\frac{x_n+x}{2}\bigg)=1$$ implies that $x_n \longrightarrow x$ ($x_n \overset{w}{\longrightarrow} x$).	 An element $x \in S_{\mathbb{X}}$ is said to be an exposed point of $B_{\mathbb{X}}$ if there exists $x^* \in J(x)$ such that $x^*(y) < 1= x^*(x),$ for any $y \in S_{\mathbb{X}} \setminus \{x\}.$ Clearly, every exposed point of $B_{\mathbb{X}}$ is also an extreme point of $B_{\mathbb{X}}$. The set of all  exposed points is  denoted by $Exp(B_{\mathbb{X}}).$ 
	We say $x \in S_{\mathbb{X}}$ to be a strongly exposed point of $B_{\mathbb{X}}$ if there exists $x^* \in J(x)$ such that for any sequence $\{x_n\} \subset  B_{\mathbb{X}},$ $x^*(x_n) \longrightarrow
	1=x^*(x)$ implies that $x_n \longrightarrow x.$  In this case, we say that $x^*$ strongly exposes $x$  and $x^*$
	is said to be a strongly exposing functional of $B_{\mathbb{X}}$.
	We write $st \mymathhyphen Exp(B_{\mathbb{X}})$ and $SE(B_{\mathbb{X}})$ for the set of strongly exposed points of $B_{\mathbb{X}}$ and the set of strongly exposing
	functionals of $B_{\mathbb{X}}$,
	respectively. A nonzero element $x^* \in B_{\mathbb{X}^*}$ is said to be weak*-strongly exposed point if there exists an $x\in S_\mathbb{X}$ such that  $x_n^*(x) \longrightarrow x^*(x)$ implies that $x_n^* \xrightarrow{\enskip w^* \enskip} x^*.$	
	For a  compact (locally compact) Hausdorff topological space $K$,  $C(K, \mathbb{X})$ ($C_0(K, \mathbb{X})$) denotes the space of continuous functions (continuous functions vanishes at infinity) from $K$ to $\mathbb{X}$.  Whenever $K$ is compact, $C_0(K, \mathbb{X})= C(K, \mathbb{X}).$
	For a function $f \in C_0(K, \mathbb{X}),$ the norm attainment set of $f,$ denoted by $M_f,$ is defined as $	M_f= \{k\in K: \|f(k)\|= \|f\|\}.$  Note that for any  $f \in S_{C_0(K, \mathbb{X})},$ $ M_f$ is non-empty and compact. 
	
	For any element $ x \in \mathbb{X}, $ and any subspace $\mathbb Y$ of $ \mathbb{X}, $  $ y_0 \in \mathbb{Y} $ is said to be a best coapproximation \cite{FF} to $ x $ out of $ \mathbb Y$ if $ \| y_0 - y \| \leq \| x - y \| $ for all $ y \in \mathbb{Y}. $ 	Given any $x\in \mathbb{{X}}$ and a subspace $\mathbb{{Y}}$ of $\mathbb{X},$ $\mathcal{R}_\mathbb{{Y}}(x)$ denotes the (possibly empty)  set of all best coapproximations to $x$ out of $\mathbb{{Y}}.$ The set $dom~\mathcal{R}_{\mathbb{Y}}$ denotes the set of all points of $x \in \mathbb{X}$ from which the best coapproximation to $x$ out of $\mathbb{Y}$ exists. A subspace $\mathbb{Y}$ is said to be coproximinal if $dom~\mathcal{R}_{\mathbb{Y}} =\mathbb{X}$ and a coproximinal subspace $\mathbb{Y}$ is said to be co-Chebyshev if $\mathcal{R}_{\mathbb{Y}}(x)$ is singleton, for each $x \in \mathbb{X}.$ To study best coapproximation problem, we use Birkhoff-James orthogonality. Given $ x, y  \in \mathbb{X}$, we say that $ x $ is Birkhoff-James orthogonal \cite{B, J}  to $ y, $ written as $ x \perp_B y, $ if $ \| x+\lambda y \| \geq \| x \|, $ for all $ \lambda \in \mathbb{K}. $ From \cite{FF} we note that $ y_0 \in \mathbb{Y} $ is a best coapproximation to $ x $ out of $ \mathbb{Y} $ if and only if 
	$$ \mathbb{{Y}} \perp_B (x- y_0), \,\, i.e., \,\, y \perp_B (x-y_0) ~  \forall  y \in \mathbb{Y}. $$ 
	
	Given $\epsilon \in [0, 1)$ and $ x, y \in \mathbb{X},$  $x$ is said to be $\epsilon$-Birkhoff-James orthogonal \cite{C} to $y,$ written as $ x \perp_B^\epsilon y,$ if 
	\[\|x + \lambda y\| \geq \|x\| -  \epsilon \|\lambda y\|~ \text{for ~every ~} ~\lambda \in \mathbb{K}. \]
	The above definition, in conjunction with the previously mentioned relation between Birkhoff-James orthogonality and the best coapproximation, naturally leads us to the following definition of  $\epsilon$-best coapproximation, introduced in \cite{SSGP3}:\\
	Let $ \epsilon \in [0,1).$  For a subspace $\mathbb Y$ and a  given $ x \in \mathbb{X},  $ we say that $ y_0 \in \mathbb{Y} $ is an  $\epsilon$-best coapproximation to $ x $ out of $ \mathbb{{Y}} $ if 
	$$ \mathbb{{Y}} \perp_B^\epsilon (x- y_0), \,\, i.e., \,\, y \perp_B^\epsilon (x-y_0) \,  \forall  y \in \mathbb{Y}. $$ 
	As noted in \cite{SSGP3},  the definitions of best coapproximation and $\epsilon$-best coapproximation motivate us to study the following two special types of subspaces of a Banach space:
	
	\begin{definition}\label{epsilon} 
		(i)  A subspace  $\mathbb{{Y}}$ of  $\mathbb{X}$ is said to be an \textit{anti-coproximinal subspace} of $\mathbb{X}$ if for any $x \in \mathbb{X} \setminus \mathbb{Y},$ there does not exist a best coapproximation to $x$ out of $\mathbb{Y}.$ Equivalently, a subspace $\mathbb{Y}$ is anti-coproximinal in $\mathbb{X}$ if  for any nonzero $x \in \mathbb{X},$ $\mathbb{Y} \not\perp_B x.$\\
		(ii) A subspace  $\mathbb{{Y}}$  of $\mathbb{X}$ is said to be a \textit{strongly anti-coproximinal subspace} of $\mathbb{X}$ if for any given $x \in \mathbb{X} \setminus \mathbb{Y}$ and for any $\epsilon \in [0, 1),$ there does not exist an $\epsilon$-best coapproximation to $x$ out of $\mathbb{Y}.$  Equivalently, a subspace  $\mathbb{Y}$ is strongly anti-coproximinal if  for any nonzero $x \in \mathbb{X}$ and for any $\epsilon \in [0, 1),$ $\mathbb{Y} \not\perp_B^{\epsilon} x.$ 
		
	\end{definition}
	
	The primary objective to introduce these two notions is to investigate the least  favorable scenario	that can arise in studying the best coapproximation problem. In \cite{SSGP3, SGSP} the extreme nature of these two subspaces have been portrayed in different class of Banach spaces including finite-dimensional polyhedral Banach space, smooth Banach space, strictly convex Banach space and $C(K)$ space. In this article we further continue our exploration in the space of vector-valued continuous function space as well as in general Banach spaces. Our main results are divided into four  sections including the introductory one.  Firstly, we provide some necessary  and sufficient conditions separately for both  anti-coproximinal and strongly anti-coproximinal subspaces in $C_0(K, \mathbb{X}).$ However, we obtain a complete characterization  for strongly anti-coproximinal subspaces in $C_0(K,\mathbb{X}),$ considering the unit of $\mathbb{X}^*$ is the closed convex hull of its weak*-strongly exposed points. We next investigate the conditions under which a subspace containing $\mathcal{F}(\mathbb{X}, \mathbb{Z})$ is anti-coproximinal and strongly anti-coproximinal in $\mathbb{L}(\mathbb{X}, \mathbb{Y}),$ where $\mathbb{Z}$ is a subspace of $\mathbb{Y}.$  In the final section, we examine these two classes of subspaces within the general framework of Banach spaces. We provide a complete characterization of anti-coproximinal and strongly anti-coproximinal subspaces, as well as coproximinal and co-Chebyshev subspaces.
	
	%\section{main results.}	
	\section{Anti-coproximinality in $C_0(K,\mathbb{X})$}

	Here we explore the (strongly) anti-coproximinal subspaces in $ C_0(K, \mathbb{X}) $. In \cite[Th.~3.35]{SGSP}, it is noted that $ C(K, \mathbb{Y}) $ is (strongly) anti-coproximinal in $ C(K, \mathbb{X}) $ if and only if $ \mathbb{Y} $ is (strongly) anti-coproximinal in $\mathbb{X}$, assuming that $K $ is compact and perfectly normal. Our aim  is to investigate conditions under which any arbitrary subspace $ \mathcal{Y} $ is (strongly) anti-coproximinal in $C_0(K,\mathbb{X}).$ For this we need the following two well known lemmas which describes orthogonality in terms of functionals.
	
	\begin{lemma}\label{James}\cite[Th. 2.1]{J}
		Let $\mathbb{X}$ be a Banach space and let $x, y \in \mathbb{X}.$  Then $x \perp_B y$ if and only if there exists $x^* \in J(x)$ such that $x^*(y)=0.$
	\end{lemma}

	\begin{lemma}\label{Chem}\cite[Th. 2.3]{CSW}
		Let $\mathbb{X}$ be a Banach space. Suppose $\epsilon \in [0,1)$ and let $x, y \in \mathbb{X}.$  Then $x \perp_B^\epsilon y$ if and only if  there exists $x^* \in J(x)$ such that $|x^*(y)| \leq \epsilon\|y\|.$
	\end{lemma}
	
	Further we need the following two theorems that characterizes orthogonality of elements in $C_0(K, \mathbb{X}).$
	
	\begin{theorem}\label{martin:continuous}\cite[Th. 3.5]{MMQRS}
		Let $K$ be a locally compact Hausdorff space and let $\mathbb{X}$ be a Banach space. Suppose   $C \subset S_{\mathbb{X}^*}$ is such that $B_{\mathbb{X}^*}= \overline{co(C)}^{w^*}.$  Then  for $f, g \in C_0(K, \mathbb{X} ),$   
		\[ f \perp_B g \iff	0 \in co\bigg(\bigg\{  y^*(g(k)): k \in K, y^* \in C,   y^*(f(k))= \|f\|\bigg\}\bigg).
		\]
	\end{theorem}

	\begin{theorem}\label{approximate:continuous}\cite[Th. 2.8]{SGSP}
		Let $K$ be a locally compact Hausdorff space and let $\mathbb{X}$ be a Banach space. Suppose  $C \subset S_{\mathbb{X}^*}$ is such that $B_{\mathbb{X}^*}= \overline{co(C)}^{w^*}.$  Then  for $f, g \in C_0(K, \mathbb{X} ),$
		\[
		f \perp_B^\epsilon g \iff
		co\bigg(\bigg\{  y^*(g(k)): k \in K, y^* \in C,   y^*(f(k))= \|f\|\bigg\}\bigg) \cap \mathcal{D}(\epsilon\|g\|) \neq \emptyset.
		\]
	\end{theorem}

	%\begin{lemma}\label{Rudin}\cite[Th. 4.7]{R}
	%	Let $\mathbb{Z}$ be a subspace of $\mathbb{X}^*.$ Then $(^\perp \mathbb{Z})^\perp= \overline{\mathbb{Z}}^{w^*}.$ 
	%\end{lemma}

	%\begin{lemma}\cite[Lemma 2.15]{KMMPQ}\label{lemma:kadet}Let $\mathbb{X}$ be a Banach space and  let $C \subseteq B_{\mathbb{X}^*}$ be such that $B_{\mathbb{X}^*} = \overline{co(C)}^{w^*}.$ Suppose $ x \in S_{\mathbb{X}}$ and $z \in  \mathbb{X}.$ Then, for every $x_0^* \in S_{\mathbb{X}^*}$ with $ x_0^*(x)=1$ and every $ \delta > 0,$ there exists $ x^* \in C$ such that  $$ \Re[x^*(x)] > 1 - \delta \quad \textit{and} \quad \Re [x^*(z)]> \Re [x_0^*(z)] - \delta.$$
	
	%\end{lemma}

	%\begin{lemma}\cite[Th. 1.6]{RS86}
	%	Let $\mathbb{X}, \mathbb{Y}$ be  Banach spaces. Then	$$w^*\mymathhyphen st \mymathhyphen Exp(B_{\mathbb{L}(\mathbb{X}, \mathbb{Y})^*})= \{ y^* \otimes x: y^* \in  w^* \mymathhyphen st \mymathhyphen Exp(B_{\mathbb{Y}^*}), x \in st\mymathhyphen Exp(B_{\mathbb{X}})\},$$	where $y^* \otimes x(T)=y^*(Tx), ~\forall T \in \mathbb{L}(\mathbb{X}, \mathbb{Y}).$\end{lemma}

%\begin{prop}	Let $\mathbb{X}$ be a Banach space and let $\mathbb{Y}$ be a subspace of $\mathbb{X}.$ Then there exists a minimal selection map $\phi$ of $\mathbb{Y}$ such that $Im~\phi \subset Ext(B_{\mathbb{X}^*}).$\end{prop}

%\section*{Section-I}

For a subspace $\mathcal{Y}$ of $C_0(K, \mathbb{X}),$ a mapping $\psi : \mathcal{Y} \to K,$ defined by $\psi(f)= k_f,$ for some $k_f \in M_f,$ is called a selection map corresponding to $\mathcal{Y}$.  Let   $$ \mathcal{T}_\mathcal{Y} = \{\psi: \mathcal{Y} \to K , \,  \text{$\psi$ is a selection map corresponding to $\mathcal{Y}$}\}.$$ 	For $\psi \in \mathcal{T}_\mathcal{Y}$,  we write $\mathcal{A}_{\psi}= \{ f(k_f):   f \in S_\mathcal{Y}, \psi(f)= k_f \}.$ Whenever the context is clear we only write $\psi$ as selection map.  Now we are in a position to present a necessary condition for (strongly) anti-coproximinal subspace in the space $C_0(K, \mathbb{X}).$ 

\begin{lemma}\label{lemma:1}
	Let $K$ be a locally compact  Hausdorff space and let $\mathcal{Y}$ be a closed proper subspace of $C_0(K, \mathbb{X})$. 
	\begin{itemize}
		
		\item[(i)] If $\mathcal{Y}$ is anti-coproximinal in $C_0(K, \mathbb{X})$ then for any $\psi \in \mathcal{T}_\mathcal{Y}$ there exists no $x \in \mathbb{X} \setminus \{0\} $ such that $\mathcal{A}_{\psi} \perp_B x.$

		\item [(ii) ] If $\mathcal{Y}$ is strongly anti-coproximinal in $C_0(K, \mathbb{X})$ then for any $\psi \in \mathcal{T}_\mathcal{Y}$  there exists no $x \in \mathbb{X} \setminus \{0\} $  and no $\epsilon \in [0,1)$ such that $\mathcal{A}_{\psi} \perp_B^\epsilon x.$ 
	\end{itemize}
\end{lemma}

\begin{proof}
	(i)      Suppose on the contrary that  there exists a $z \in S_\mathbb{X}$ such that $\mathcal{A}_{\psi} \perp_B z$ for some   $\psi \in   \mathcal{T}_\mathcal{Y}$. Let $\alpha \in S_{C_0(K)}$. Define $g : K \to \mathbb{X}$ such that $g(k)= \alpha(k) z.$ Clearly, $ g \in C_0(K, \mathbb{X})$ and $\|g\|=1.$  Let $f \in S_\mathcal{Y}$ and $k_f \in Im ~\psi.$ Then $f(k_f) \in \mathcal{A}_{\psi}.$ As $\mathcal{A}_{\psi} \perp_B z,$ we have $f(k_f) \perp_B z.$ Following Lemma \ref{James}, there exists $y_1^* \in S_{\mathbb{X}^*}$ such that $y_1^*(f(k_f))=\|f\|$ and $y_1^*(z)=0.$ This implies $y_1^*(g(k_f))=0.$ 
	Therefore
	\[ 0 \in co\bigg(\bigg\{  y^*(g(k)): k \in K, y^* \in C,   y^*(f(k))= \|f\|\bigg\}\bigg).\]
	From Theorem \ref{martin:continuous}, $f \perp_B g$ and consequently $\mathcal{Y} \perp_B g.$ This contradicts that $\mathcal{Y}$ is anti-coproximinal in $C_0(K, \mathbb{X}).$

	(ii) follows using similar arguments as (i).
\end{proof}

\begin{remark}
	Let  $\mathcal{Y}$ be a  (strongly) anti-coproximinal subspace of  $C_0(K, \mathbb{X}).$ Then either $span~\mathcal{A}_{\psi} = \mathbb{X}$ or $ span~ \mathcal{A}_{\psi}$ is  (strongly) anti-coproximinal in $\mathbb{X}.$
\end{remark}
In the following theorem we present a necessary condition for strongly anti-coproximinal subspace in terms of w-ALUR and ALUR points.

\begin{theorem}\label{prop:C(K,X)}
	Let $K$ be a locally compact  Hausdorff space   and let $\mathcal{Y}$ be a strongly anti-coproximinal subspace of $C_0(K, \mathbb{X})$. %such that  $ \{ f(k):   f \in \mathcal{Y}, k \in M_f\} \neq \mathbb{X}.$
	Then for any $\psi \in \mathcal{T}_{\mathcal{Y}},$
	\begin{itemize}
		\item[(i)]  $ \overline{\mathcal{A}_{\psi}}^w $ contains all w-ALUR points of $B_{\mathbb{X}}. $ 
		
		\item[(ii)]  $ \overline{\mathcal{A}_{\psi}} $ contains all ALUR points of $B_{\mathbb{X}}. $ 
		
	\end{itemize}
\end{theorem}

\begin{proof}
	(i) Suppose on the contrary that $ w \in S_{\mathbb{X}}$ is a w-ALUR point of $B_{\mathbb{X}}$ such that $w \notin \overline{\mathcal{A}_{\psi}}^w$ for some $\psi \in \mathcal{T}_{\mathcal{Y}}.$
	As $\mathcal{Y}$ is strongly anti-coproximinal, following Lemma \ref{lemma:1} we get that there exists no $x \in \mathbb{X} \setminus\{0\}$ and no $\epsilon \in[0,1)$  such that $\mathcal{A}_{\psi} \perp_B^{\epsilon} x.$   This implies for any $\epsilon_n \in [0,1) \to 1, $ there exists $ f_n(k_n) \in \mathcal{A}_{\psi}$ such that $f_n(k_n) \not\perp_B^{\epsilon_n} w.$ Let $x_n^* \in J(f_n(k_n)),~ \forall n \in \mathbb{N}.$ Following Lemma \ref{Chem} we get $ |x_n^*(w)| > \epsilon_n \|w\|.$ Since $B_{\mathbb{X}^*}$ is weak*-compact,
	it follows that there exists a weak*-cluster point $x^* \in B_{\mathbb{X}^*}$ of the sequence $\{x_n^*\}.$
	This implies $x^*(w)$ is a cluster point of $\{x_n^*(w)\}. $ Following \cite[ Th. 8 (p. 72)]{kelly}, there exists a subsequence $\{x_{n_r}^*(w)\}$ of $\{x_n^*(w)\}$ such that $x_{n_r}^*(w) \to x^*(w).$ As $|x_{n_r}^*(w)| > \epsilon_{n_r}\|w\|,$ for each $r \in \mathbb{N},$ therefore $|x^*(w)| \geq \|w\|,$ which in turn implies that $|x^*(w)|= \|w\|.$ This implies $x^* \in J(w) $ or $x^* \in J(-w).$ Suppose that $x^* \in J(w).$ Note that
	$$ x_{n_r}^* \bigg(\frac{f_n(k_n) +w}{2}\bigg)=  \frac{x_{n_r}^*(f_n(k_{n_r}))}{2} + \frac{x_{n_r}^*(w)}{2}= \frac{1}{2}+\frac{1}{2} x_{n_r}^*(w) ,$$ 
	and  so $\lim  x_{n_r}^* (\frac{f_n(k_{n_r})+w}{2})=1.$ Also,
	$w$ being an w-ALUR point, we get that $f_n(k_{n_r}) \overset{w}{\longrightarrow} w.$ This contradicts the fact that $w \in \overline{\mathcal{A}_{\psi}}^w.$ If $x^* \in J(-w)$ then proceeding similarly we get $\lim x^*(\frac{f_n(n_{k_r}) - w}{2}) = 1$ and consequently, $-w \in \overline{\mathcal{A}_{\psi}}^w,$ again a contradiction. This completes the proof. \\
	%Let $\alpha \in C_0(K)$ and let $f: K \to \mathbb{X}$ such that $f(k) = \alpha(k) w.$

	(ii) follows  similarly  as (i).
\end{proof}

In case the space $\mathbb{X}$ is finite-dimensional, we have the following result.

\begin{theorem}\label{rotund}
	Let $K$ be a locally compact  Hausdorff space and let $\mathbb{X}$ be finite-dimensional.  Suppose that $\mathcal{Y}$ is a strongly anti-coproximinal subspace of $C_0(K, \mathbb{X})$ and 
	$\psi \in \mathcal{T}_{\mathcal{Y}}.$	Then the following holds:
	\begin{itemize}
		\item[(i)] Let $F$ be a  maximal face of $B_{\mathbb{X}}$ and $w \in int(F).$ Then there exists $u \in \overline{\mathcal{A}_{\psi}}$ such that    $J(w) \cap J(u) \neq \emptyset.$ 
		\item [(ii)]     For any maximal face $F$ of $B_{\mathbb{X}}$,   $ \overline{\mathcal{A}_{\psi}} \cap F \neq \emptyset.$

		\item[(iii)]   $\overline{\mathcal{A}_{\psi}}$ contains all rotund points of $B_{\mathbb{X}}.$
	\end{itemize}
\end{theorem}

\begin{proof}
	(i) Let $\psi \in \mathcal{T}_{\mathcal{Y}}$. Let $F$ be a maximal face and let $w \in int(F)$. As $\mathcal{Y}$ is strongly anti-coproximinal, following Lemma \ref{lemma:1} we get that there exists no $x \in \mathbb{X} \setminus\{0\}$ and no $\epsilon \in[0,1)$  such that $\mathcal{A}_{\psi} \perp_B^{\epsilon} x.$ 
	Therefore, for any $\epsilon_n \in [0,1) \to 1, $ there exists $ f_n(k_n) \in \mathcal{A}_{\psi}$ such that $f_n(k_n) \not\perp_B^{\epsilon_n} w.$ Let $x_n^* \in J(f_n(k_n)),~ \forall n \in \mathbb{N}.$ From Lemma \ref{Chem}, we get $ |x_n^*(w)| > \epsilon_n \|w\|.$ As $B_{\mathbb{X}^*}$ is weak*-compact,
	it follows that there exists a weak*-cluster point $x^* \in B_{\mathbb{X}^*}$ of the sequence $\{x_n^*\}.$
	This implies $x^*(w)$ is a cluster point of $\{x_n^*(w)\}. $ Following \cite[ Th. 8 (p. 72)]{kelly}, there exists a subsequence $\{x_{n_r}^*(w)\}$ of $\{x_n^*(w)\}$ such that $x_{n_r}^*(w) \to x^*(w).$ As $|x_{n_r}^*(w)| > \epsilon_{n_r}\|w\|,$ for each $r \in \mathbb{N},$ therefore $|x^*(w)| \geq \|w\|,$ which in turn implies that $|x^*(w)|= \|w\|.$
	This implies either  $x^* \in J(w)$ or $-x^* \in J(w).$ Since $S_\mathbb{X}$ is compact, it follows that $f_{n_r}(k_{n_r}) \to u,$ for some $u \in S_\mathbb{X}$. Clearly, $u \in \overline{\mathcal{A}_{\psi}}.$ So for any $r \in \mathbb{N},$
	\begin{eqnarray*}
		|x_{n_r}^*(f_{n_r}(k_{n_r})) - x^*(u)| &=&  |x_{n_r}^*(f_{n_r}(k_{n_r})) - x_{n_r}^*(u) + x_{n_r}^*(u)- x^*(u)|\\
		&\leq & \|x_{n_r}^*\| \|f_{n_r}(k_{n_r})- u\| + \| x_{n_r}^*(u)- x^*(u)\|.
	\end{eqnarray*}
	Taking $r \to \infty,$ in the above relation we get  $x_{n_r}^*(f_{n_r}(k_{n_r})) \longrightarrow x^*(u).$ Since for each $ r,$ $x_{n_r}^*(f(k_{n_r}))=1,$ we have $x^*(u)=1.$ In other words, $x^* \in J(u).$ So, whenever $x^* \in J(w),$ we have $J(w) \cap J(u) \neq \emptyset$. Again whenever $-x^* \in J(w)$ we have $-x^* \in J(-u)$ and therefore $J(w) \cap J(-u) \neq \emptyset.$  \\
	
	(ii)  Let $F$ be a maximal face of $B_{\mathbb{X}}$ and let $w \in int(F).$ From (i),	we get an element $u \in \overline{\mathcal{A}_{\psi}}$ such that $J(w) \cap J(u) \neq \emptyset.$ Suppose that  $x^* \in J(w) \cap J(u).$  As $w \in int(F)$ and $x^*(w)=1,$ it is easy to check that for any $v \in F,$ $x^*(v)=1.$
	Let $x \in co(F \cup \{u\}).$ Then $x= (1-t) z+ tu,$ for some $t \in [0,1]$ and $z \in F.$ The relation $x^*(u)=x^*(z)=1$ yields that $x^*(x)=1.$ This implies that $x \in S_{\mathbb{X}}.$ So, $co(F \cup \{u\}) \subset S_{\mathbb{X}}.$ Therefore, there exists some face $F' $ of $B_{\mathbb{X}}$ such that $F \subset co(F \cup \{u\}) \subset F'.$ As $F$ is a maximal face, we have $F=co(F \cup \{u\}),$ so $u \in F.$ Therefore, $\overline{\mathcal{A}_{\psi}} \cap F \neq \emptyset,$ as desired.\\
	
	(iii) follows from (ii) considering the fact that every rotund point is a maximal face of the unit ball.
	
\end{proof}

Further down the line if $\mathbb{X}$ is a finite-dimensional real polyhedral Banach space, then we show that each $\mathcal{A}_{\psi}$ intersects interior of every maximal face of $B_{\mathbb{X}},$ for any $\psi \in \mathcal{T}_{\mathcal{Y}}.$

\begin{theorem}\label{prop:polyhedral}
Let $K$ be a locally compact Hausdorff space and let $\mathbb{X}$ be a finite-dimensional polyhedral Banach space. Suppose that $\mathcal{Y}$ is strongly anti-coproximinal in $C_0(K, \mathbb{X}).$ Then for any $\psi \in \mathcal{T}_{\mathcal{Y}}$,  $\mathcal{A}_{\psi}$ intersects interior of every maximal face of $B_{\mathbb{X}}$. 
\end{theorem}

\begin{proof}
Let $\pm F_1, \pm F_2, \ldots, \pm F_r$ are the maximal faces of $B_{\mathbb{X}}$ and let $y_i^* \in Ext(B_{\mathbb{X}^*})$ corresponds to the face $F_i$ (see \cite[Lemma 2.1]{SPBB}).  Suppose on the contrary that there exists $\psi \in \mathcal{T}_{\mathcal{Y}}$ such that  $\mathcal{A}_{\psi} $ does not  intersect interior of a maximal face $F_j$ of $B_{\mathbb{X}}.$  Let $u \in int(F)$ and $\alpha \in S_{C_0(K)}.$ Observe that $  J(u)=\{y_j^*\}.$ Define $g: K \to \mathbb{X}$ such that $g(k)= \alpha(k) u.$ So, $ g \in S_{C_0(K, \mathbb{X})}.$    Let $$\epsilon = \max \{| y_i^*(u)|: i \in \{1, 2, \ldots, r\} \setminus \{j\}  \}. $$
Clearly $\epsilon < 1.$
We claim that $\mathcal{Y} \perp_B^\epsilon g.$ Let $f \in \mathcal{Y}$ and let $\psi(f)=k_f$. Suppose that $f(k_f) = x_f \in S_{\mathbb{X}}.$ Clearly, $x_f \in \mathcal{A}_{\psi}.$ As $\mathcal{A}_\psi \cap int(F_j) = \emptyset,$ there exists $t \in \{1, 2, \ldots, r\} \setminus \{j\}$ such that $x_f \in F_t.$ Therefore, $ y_t^* \in J(x_f) \cap \{Ext(B_{\mathbb{X}^*}) \setminus \{ y_j^*\}\}.$ As $|y_t^*(u)| = |y_t^*(g(k_f))| \leq \epsilon, $
\[
y_t^*(u) \in co \bigg( \bigg\{ y^*(g(k)) : k \in K, y^* \in V, y^*(f(k)) = \|f\| \bigg\} \bigg) \cap \mathcal{D}(\epsilon).
\]
Following Theorem \ref{martin:continuous}, we get $f \perp_B^\epsilon g.$ Consequently, $\mathcal{Y} \perp_B^\epsilon g.$ This contradicts that $\mathcal{Y}$ is strongly anti-coproximinal in $C_0(K, \mathbb{X}).$

(ii) $\implies (i): $
\end{proof}

We illustrate the above theorem with the following example.
\begin{example}
Let $K=[0, 1]$ and $\mathbb{X}=\ell_1^3.$ Consider the functions $f_1, f_2: [0, 1]\to \ell_1^3$ defined  as $f_1(t)=(t+1, t, 0)$ and $f_2(t)=(t^2, 0, 0).$ Clearly, $f_1, f_2$ are linearly independent. Let $\mathcal{Y}=span \{f_1, f_2\}.$ For any $\psi \in \mathcal{T}_\mathcal{Y}$, $\mathcal{A}_{\psi} \subset span\{(x, y, 0): x, y \in \mathbb R\}\subset \ell_1^3.$ Therefore, $\mathcal{A}_{\psi}$ does not intersects interior of any maximal face of $B_{\ell_1^3}.$ Following Theorem \ref{prop:polyhedral}, $\mathcal{Y}$ is not strongly anti-coproximinal in $C([0,1], \ell_1^3).$
\end{example}

Next we present a necessary condition for anti-coproximinal subspaces of $C_0(K, \mathbb{X}).$ 

\begin{prop}\label{prop:anti}
Let $K$ be a locally compact normal space and $\mathcal{Y}$ be an anti-coproximinal subspace of $C_0(K, \mathbb{X}).$ Then for any nonempty open set $U \subset K$ there exists $f \in \mathcal{Y}$ such that $M_f \subset U.$
\end{prop}

\begin{proof}
Suppose on the contrary there exists a nonempty open set $U \subset K$ such that for any $f \in \mathcal{Y},$ $ M_f \cap (K \setminus U) \neq \emptyset$. Let $k_0 \in U$ and $x \in S_{\mathbb{X}}.$ Following Uryshonn's Lemma \cite{Munkres} there exists a continuous function $\alpha: K \to [0,1]$ such that $\alpha(k_0)=1, \alpha(k)=0 ~ \forall k \in K \setminus U.$ Let $g : K \to \mathbb{X}$ such that $g(k)= \alpha(k)x.$  Clearly $ g(k_0)= x$ and $g(k)=0 ~ \forall k \in K \setminus U.$   We claim  that $\mathcal{Y} \perp_B g.$ Let $f \in \mathcal{Y}$ and $k_f \in M_f \cap (K \setminus U).$ Observe that $g(k_f)=0.$ Therefore, 
\[ 0 \in co\bigg(\bigg\{  y^*(g(k)): k \in K, y^* \in S_{\mathbb{X}^*},   y^*(f(k))= \|f\|\bigg\}\bigg).\]
From Theorem \ref{martin:continuous}, $f \perp_B g$ and consequently  $\mathcal{Y} \perp_B g.$ This contradicts  that $\mathcal{Y}$ is  anti-coproximinal in $C_0(K, \mathbb{X}).$
\end{proof}

Now we provide a sufficient condition for strongly anti-coproximinal subspace of $C_0(K, \mathbb{X}).$

\begin{theorem}\label{sufficient; strong}
Let $K$ be a locally compact Hausdorff space.  Suppose that for any nonempty open set $U $of $ K$ and  for any nonempty  weak*-open set $V $ of $\mathbb{X}^*$ containing an extreme point of $B_{\mathbb{X}^*}$, there exists $f \in \mathcal{Y}$ such that $M_f \subset U$ and $J(f(k))\subset V ~\forall k \in M_f.$ Then $\mathcal{Y}$ is strongly anti-coproximinal in $C_0(K, \mathbb{X}).$
\end{theorem}

\begin{proof}
Suppose on the contrary that $\mathcal{Y}$ is not strongly anti-coproximinal in $C_0(K, \mathbb{X})$. In other words, there exists $\epsilon \in [0,1), g \in S_{C_0(K, \mathbb{X})}$ such that $\mathcal{Y} \perp_B^\epsilon g.$ As $M_g$ is nonempty, suppose $\|g(k')\|=\|g\|=1,$ for some $k'\in K.$ Let $z^* \in J(g(k')) \cap Ext(B_{\mathbb{X}^*}).$ Since $z^*$ and $g$ are continuous,  the set $U= \{u\in K: \Re[z^* (g(u))]>\frac{1+\epsilon}{2}\}$ is open in $K.$  Clearly $U$ is nonempty, as $k' \in U.$ Now for each $u\in U,$ consider the set with respect to $z^*$ as 
$$V_{u}=\bigg\{y^* \in \mathbb{X}^*: |y^*(g(u))- z^*(g(u)) | <\frac{1-\epsilon}{2}\bigg\}.$$
Clearly, it is a weak*-open set containing $z^*.$ Therefore, the set $V=\bigcup_{u\in U} V_u$ is also a weak*-open in $\mathbb{X}^*$ containing $z^*.$ 
Suppose that $y^*\in V$ and $k\in U$ are chosen arbitrarily. Then 
\begin{eqnarray*}
	\Re [y^*(g(k))] &=& \Re [z^*(g(k))]- \Re [z^*(g(k))] + \Re [y^*(g(k))]\\
	&=&   \Re [z^*(g(k))]-  \Re [z^*(g(k)) -y^*(g(k))]\\
	&>& \frac{1+\epsilon}{2}-\frac{1-\epsilon}{2}=\epsilon.
\end{eqnarray*}
This shows that for any $y^* \in V$ and any $k \in U,$ we have $\Re [y^*(g(k))]>\epsilon.$ Now from hypothesis there exists an $f\in \mathcal{Y}$ such that $M_f\subset U$ and $J(f(k)) \subset V,$ for all $k\in M_f.$ 
Since $f\perp_B^{\epsilon} g,$ it follows that
$$\mathcal{D}(\epsilon)\cap co \bigg(\bigg\{y^*(g(k)): k\in K, y^*\in S_{\mathbb{X}^*}, y^*(f(k))=\|f\|\bigg\} \bigg) \neq \emptyset.$$ 
As $M_f \subset U$ and $J(f(k)) \subset V,$ 
$$\mathcal{D}(\epsilon)\cap co \bigg( \bigg\{y^*(g(k)): k\in U, y^*\in V, y^*(f(k))=\|f\|\bigg\} \bigg) \neq \emptyset.$$ 
This implies  there exist $k_i \in M_f \subset U, x_i^* \in V$ and $t_i > 0$ for $i \in \{1,2,3\}$ such that 
$\sum_{i=1}^3 t_i =1$, $x_i^* (f(k_i))= \|f\|$ and $| \sum_{i=1}^3 t_i x_i^*(g(k_i))| \leq  \epsilon$. Therefore, $$ \sum_{i=1}^3 t_i \Re[ x_i^*(g(k_i)) ] \leq \epsilon.$$ But for any $u \in U, y^* \in V, \Re[y^*(g(u))] > \epsilon,$ which is a contradiction.  Thus we get our desired result. 
\end{proof}

In the following theorem, we  obtain a characterization of the strongly anti-coproximinal subspaces of $C_0(K, \mathbb{X})$ under the condition that $B_{\mathbb{X}^*}= \overline{co(w^* \mymathhyphen st \mymathhyphen Exp(B_{\mathbb{X}^*}))}^{w^*}.$ To do so we need the following lemma, which can be obtained directly from \cite[Lemma 2.15]{KMMPQ}.

\begin{lemma}\label{lemma:kadet}
Let $\mathbb{X}$ be a Banach space such that $B_{\mathbb{X}^*} = \overline{co(w^* \mymathhyphen st \mymathhyphen Exp(B_{\mathbb{X}^*}))}^{w^*}.$ Let  $ x \in S_{\mathbb{X}}$. Then, for every $x_0^* \in J(x)$ and every $ \delta > 0,$ there exists $ x^* \in w^* \mymathhyphen st \mymathhyphen Exp(B_{\mathbb{X}^*})$ such that  $ \Re[x^*(x)] > 1 - \delta.$
\end{lemma}

\begin{theorem}\label{theorem:C(K,X)}
Let $K$ be a locally compact normal space. Let $\mathbb{X}$ be a Banach space such that $B_{\mathbb{X}^*}= \overline{co(w^* \mymathhyphen st \mymathhyphen Exp(B_{\mathbb{X}^*}))}^{w^*}.$ Then the following are equivalent:
\begin{itemize}
	\item [(i)] A subspace $\mathcal{Y}$ is strongly anti-coproximinal in $C_0(K, \mathbb{X}).$
	\item[(ii)] For any nonempty open set $U \subset K$ and for any  nonempty weak*-open set $ V \subset S_{\mathbb{X}^*}$ containing a weak*-strongly exposed point,  there exists $f \in \mathcal{Y}$ such that $M_f \subset U$ and $J(f(k)) \subset V ~\forall k \in M_f.$
\end{itemize}
\end{theorem}

\begin{proof}
(i) $\implies$ (ii):    Suppose on the contrary that atleast one of the following is true:
\begin{itemize}
	\item [(a)] there exists a nonempty open set $U \subset K$ such that  for any $f \in \mathcal{Y},$ $ M_f \cap (K \setminus U) \neq \emptyset$
	
	\item[(b)] there exists a nonempty weak*-open set  $V_0 \subset S_{\mathbb{X}^*}$ containing a weak*-strongly exposed point $v_0^*$ such that for any $f \in \mathcal{Y},$  $J(f(k_f)) \cap (S_{\mathbb{X}^*} \setminus V_0 ) \neq \emptyset,$ for some $k_f \in M_f.$
\end{itemize}

Let us first assume (a) holds true.  Then proceeding similarly  as in Proposition \ref{prop:anti} we get a contradiction.

Now we assume (b) holds true.
As $v_0^*$ is a weak*-strongly exposed point of $B_{\mathbb{X}^*}, $ there exists $v \in S_{\mathbb{X}}$ such that $v_n^*(v) \longrightarrow v_0^*(v) \implies v_n^* \xrightarrow{\enskip w^* \enskip}  v_0^*,$ for any sequence $\{v_n^*\} \subset S_{\mathbb{X}^*}.$ Let 
$$\sup \bigg\{ x^*( v): x^* \in S_{\mathbb{X}^*} \setminus V_0   \bigg\}= \epsilon.$$ 
We claim that $\epsilon < 1.$ If $\epsilon=1,$ then there exists a sequence $\{x_n^*\} \subset S_{\mathbb{X}^*} \setminus V_0$ such that $x_n^*(v) \longrightarrow 1 = v_0^*(v).$  Therefore, $x_n^* \xrightarrow{\enskip w^* \enskip}  v_0^* \in V_0,$ a contradiction as  $V_0$ is  weak*-open set in $S_{\mathbb{X}^*}.$ 
Let $k_0 \in K$ and $U$ be an open set such that $k_0 \in U \subset K.$
Following Uryshonn's Lemma \cite{Munkres} there exists $\beta: K \to [0,1]$ such that $\beta(k_0) =1$ and $\beta(k)=0~ \forall k \in K \setminus U.$ Let $g: K \to \mathbb{X}$ such that $g(k) = \beta (k) v,~\forall k \in K.$ Clearly, $g \in C_0(K, \mathbb{X}).$ We show that $\mathcal{Y} \perp_B^\epsilon g.$ Let $f \in \mathcal{Y}$. Suppose  $ k_f \in M_f $ and $y_f^* \in J(f(k_f)) \cap (S_{\mathbb{X}^*} \setminus V_0).$ Clearly, $y_f^*(v) \leq \epsilon,$ and so $y_f^*(g(k_f)) \leq \epsilon.$ Therefore, 
\[ \mathcal{D}(\epsilon) \cap co\bigg(\bigg\{  y^*(g(k)): k \in K, y^* \in C,   y^*(f(k))= \|f\|\bigg\}\bigg) \neq \emptyset.
\]
Following Theorem \ref{approximate:continuous}, $f \perp_B^\epsilon g.$ Consequently, $\mathcal{Y} \perp_B^\epsilon g$ and this contradicts the fact that $\mathcal{Y}$ is strongly anti-coproximinal. This completes the necessary part of the theorem.\\

(ii) $\implies$ (i): 	Suppose on the contrary that $\mathcal{Y}$ is not strongly anti-coproximinal. In other words, there exists $\epsilon \in [0,1), g \in S_{C_0(K, \mathbb{X})}$ such that $\mathcal{Y} \perp_B^\epsilon g.$ As $M_g$ is nonempty, suppose $\|g(k')\|=\|g\|=1,$ for some $k'\in K.$ Let $x^* \in J(g(k')).$ As $B_{\mathbb{X}^*}=\overline{co(w^*\mymathhyphen st \mymathhyphen Exp(B_{\mathbb{X}^*}))}^{w^*}$, following Lemma \ref{lemma:kadet} there exists $z^* \in w^*\mymathhyphen st \mymathhyphen Exp(B_{\mathbb{X}^*})$ such that $ \Re[z^*(g(k'))] > \frac{1+\epsilon}{2}.$ Since $z^*$ and $g$ are continuous,  the set $U= \{u\in K: \Re[z^* (g(u))]>\frac{1+\epsilon}{2}\}$ is open in $K.$  Clearly $U$ is nonempty, as $k' \in U.$  We now define the weak*-open set $V$ in precisely the same manner as in Theorem \ref{sufficient; strong} and proceeding similarly, we arrive at a contradiction. This completes the proof of the theorem.
\end{proof}

Recall that a Banach space $\mathbb{X}$ is said to be weakly compactly
generated (WCG) \cite{P} if there exists a weakly compact subset of $\mathbb{X}$ whose
linear span is dense in $\mathbb{X}$. Following \cite[Cor. 11]{P}, whenever $\mathbb{X}^*$ is WCG then $B_{\mathbb{X}^*}= \overline{co(w^* \mymathhyphen st \mymathhyphen Exp(B_{\mathbb{X}^*}))}^{w^*}$ and therefore, the following is immediate.  

\begin{cor}
Let $K$ be a locally compact normal space. Let $\mathbb{X}$ be a Banach space such that $\mathbb{X}^*$ is WCG.  Then the following are equivalent:
\begin{itemize}
	\item [(i)] A subspace $\mathcal{Y}$ is strongly anti-coproximinal in $C_0(K, \mathbb{X}).$
	\item[(ii)] For any nonempty open set $U \subset K$ and for any  nonempty weak*-open set $ V \subset S_{\mathbb{X}^*}$ containing a weak*-strongly exposed point,  there exists $f \in \mathcal{Y}$ such that $M_f \subset U$ and $J(f(k)) \subset V ~\forall k \in M_f.$\\
\end{itemize}
\end{cor}

Note that whenever $|K|=n$ , $C_0(K, \mathbb{X})=C(K, \mathbb{X}) = \ell_{\infty}^n(\mathbb{X})$ and therefore, the following corollary is immediate.

\begin{cor}\label{cor:directsum}
Let $\mathbb{X}$ be  a Banach space such that $B_{\mathbb{X}^*}= \overline{co(w^* \mymathhyphen st \mymathhyphen Exp(B_{\mathbb{X}^*}))}^{w^*} $. Then a subspace $\mathcal{Y} $ is strongly anti-coproximinal in $\ell_{\infty}^n(\mathbb{X})$ if and only if for any  $ k\in \{ 1, 2, \ldots, n\}$ and for any nonempty weak*-open set $V \subset \mathbb{X}^*$ containing a weak*-strongly exposed point of $B_{\mathbb{X}^*}$,  there exists $\widetilde{y} = (y_1, y_2, \ldots, y_n)\in \mathcal{Y}$ such that $ \|y_k\| > \|y_i\|\, \forall k \neq i$ and $ J(y_k) \subset V.$
\end{cor}

For finite-dimensional Banach space, every  exposed point is a weak*-strongly exposed point and $B_{\mathbb{X}^*}= \overline{co(Exp(B_{\mathbb{X}^*}))}^{w^*}.$ Hence we obtain the following corollary.

\begin{cor}
Let $K$ be a locally compact normal space. Then the following are equivalent:
\begin{itemize}
	\item [(i)] A subspace $\mathcal{Y}$ is strongly anti-coproximinal in $C_0(K, \mathbb{X}).$
	\item[(ii)] For any nonempty open set $U \subset K$ and for any  nonempty open set $ V \subset \mathbb{X}^*$ containing an exposed point of $B_{\mathbb{X}^*}$,  there exists $f \in \mathcal{Y}$ such that $M_f \subset U$ and $J(f(k)) \subset V ~\forall k \in M_f.$
\end{itemize}
\end{cor}

Considering the set $w^* \mymathhyphen st \mymathhyphen Exp(B_{\mathbb{X}^*})$ is isolated, we can get the following simpler characterization of strongly anti-coproximinal subspaces in $C_0(K, \mathbb{X}).$

\begin{theorem}\label{theorem:isolated}
Let $K$ be a locally compact normal space. Let $\mathbb{X}$ be a Banach space such that $B_{\mathbb{X}^*}= \overline{co(w^* \mymathhyphen st \mymathhyphen Exp(B_{\mathbb{X}^*}))}^{w^*}$ and the set $w^* \mymathhyphen st \mymathhyphen Exp(B_{\mathbb{X}^*})$ is isolated.  Then the following are equivalent:
\begin{itemize}
	\item [(i)] A subspace $\mathcal{Y}$ is strongly anti-coproximinal in $C_0(K, \mathbb{X}).$
	\item[(ii)] For any nonempty open set $U \subset K$ and for any $x^* \in w^* \mymathhyphen st \mymathhyphen Exp(B_{\mathbb{X}^*})$,  there exists $f \in \mathcal{Y}$ such that $M_f \subset U$ and $J(f(k)) =\{x^*\} ~\forall k \in M_f.$
\end{itemize}
\end{theorem}

\begin{proof}
(i) $\implies$ (ii):  Suppose on the contrary that atleast one of the following is true: 

\begin{itemize}
	\item[(a)] There exists a nonempty open set $U \subset K$ such that  for any $f \in \mathcal{Y},$ $ M_f \cap (K \setminus U) \neq \emptyset$.
	
	\item[(b)] There exists  $x_0^* \in w^* \mymathhyphen st \mymathhyphen Exp(B_{\mathbb{X}^*})$ such that for any $f \in \mathcal{Y},$  $J(f(k_f)) \cap (S_{\mathbb{X}^*} \setminus \{x_0^*\} ) \neq \emptyset,$ for some $k_f \in M_f.$
\end{itemize}

Let us first assume (a) holds true.  Then proceeding similar arguments as in Proposition \ref{prop:anti} we get a contradiction.  

Now we assume (b) holds true.
As $x_0^*$ is a weak*-strongly exposed point of $B_{\mathbb{X}^*}, $ there exists $x \in S_{\mathbb{X}}$ such that $x_n^*(x) \longrightarrow 1 \implies x_n^* \xrightarrow{\enskip w^* \enskip}  x_0^*,$ for any sequence $\{x_n^*\} \subset S_{\mathbb{X}^*}.$ Let 
$$\sup \bigg\{ z^*( x): z^* \in S_{\mathbb{X}^*} \setminus \{x_0^*\}   \bigg\}= \epsilon.$$ 
We claim that $\epsilon < 1.$ If $\epsilon=1,$ then there exists a sequence $\{z_n^*\} \subset S_{\mathbb{X}^*} \setminus \{x_0^*\}$ such that $z_n^*(x) \longrightarrow 1 = x_0^*(x).$  Therefore, $z_n^* \xrightarrow{\enskip w^* \enskip}  x_0^*,$ a contradiction to the fact that the set $w^* \mymathhyphen st \mymathhyphen Exp(B_{\mathbb{X}^*})$ is isolated. 
Let $k_0 \in K$ and $U$ be an open set such that $k_0 \in U \subset K.$ As $K$ is normal, from Uryshonn's Lemma \cite{Munkres} there exists   $\beta: K \to [0,1]$ such that $\beta(k_0) =1$ and $\beta(k)=0~ \forall k \in K \setminus U.$ Let $g: K \to \mathbb{X}$ such that $g(k) = \beta (k) x,~\forall k \in K.$ Clearly, $g \in C_0(K, \mathbb{X}).$ We show that $\mathcal{Y} \perp_B^\epsilon g.$ Let $f \in \mathcal{Y}$. Suppose  $ k_f \in M_f $ and $x_f^* \in J(f(k_f)) \cap (S_{\mathbb{X}^*} \setminus \{x_0^*\}).$ Clearly, $x_f^*(x) \leq \epsilon,$ and so $x_f^*(g(k_f)) \leq \epsilon.$ Therefore, 
\[ 
x_f^*(g(k_f)) \in \mathcal{D}(\epsilon) \cap co\bigg(\bigg\{  y^*(g(k)): k \in K, y^* \in C,   y^*(f(k))= \|f\|\bigg\}\bigg).
\]
Following Theorem \ref{approximate:continuous}, $f \perp_B^\epsilon g.$ Consequently, $\mathcal{Y} \perp_B^\epsilon g$ and this contradicts the fact that $\mathcal{Y}$ is strongly anti-coproximinal. This completes the necessary part of the theorem.\\

(ii) $\implies$ (i): 	Suppose on the contrary that $\mathcal{Y}$ is not strongly anti-coproximinal. In other words, there exists $\epsilon \in [0,1), g \in S_{C_0(K, \mathbb{X})}$ such that $\mathcal{Y} \perp_B^\epsilon g.$ As $M_g$ is nonempty, suppose $\|g(k')\|=\|g\|=1,$ for some $k'\in K.$ Let $x^* \in J(g(k')).$ As $B_{\mathbb{X}^*}=\overline{co(w^*\mymathhyphen st \mymathhyphen Exp(B_{\mathbb{X}^*}))}^{w^*}$, following Lemma \ref{lemma:kadet} there exists $z^* \in w^*\mymathhyphen st \mymathhyphen Exp(B_{\mathbb{X}^*})$ such that $ \Re[z^*(g(k'))] > \epsilon.$ Since $z^*$ and $g$ are continuous,  the set
$$U= \{u\in K: \Re[z^* (g(u))]> \epsilon \}$$
is open in $K.$  Clearly $U$ is nonempty, as $k' \in U.$ From hypothesis there exists $f \in \mathcal{Y}$ such that $M_f \subset U$ and $J(f(k)) =\{z^*\}, \forall k \in M_f.$ 
Since $f\perp_B^{\epsilon} g,$ it follows that
$$\mathcal{D}(\epsilon)\cap co \bigg(\bigg\{y^*(g(k)): k\in K, y^*\in S_{\mathbb{X}^*}, y^*(f(k))=\|f\|\bigg\} \bigg) \neq \emptyset.$$ 
As $M_f \subset U$ and $J(f(k)) =\{z^*\},$ we get
$$\mathcal{D}(\epsilon)\cap co \bigg( \bigg\{z^*(g(k)): k\in U,  z^*(f(k))=\|f\|\bigg\} \bigg) \neq \emptyset.$$ 
This implies  there exist $k_i \in M_f \subset U$ and $t_i > 0$ for $i \in \{1,2,3\}$ such that 
$\sum_{i=1}^3 t_i =1$, $z^* (f(k_i))= \|f\|$ and $| \sum_{i=1}^3 t_i z^*(g(k_i))| \leq  \epsilon$. Therefore, $$ \sum_{i=1}^3 t_i \Re[ z^*(g(k_i)) ] \leq \epsilon.$$ But for any $u \in U,  \Re[z^*(g(u))] > \epsilon,$ which is a contradiction.  Thus we get our desired result. 
This completes the proof of the theorem.
\end{proof}

For a finite-dimensional real polyhedral Banach space $\mathbb{X}$, every extreme point of $B_{\mathbb{X}^*}$ is an exposed point and the set of all extreme point is finite, so the following corollary is immediate.

\begin{cor}\label{cor:polyhedral}
Let $K$ be a locally compact normal space and let $\mathbb{X}$ be a finite-dimensional real polyhedral Banach space.  Then the followings are equivalent:
\begin{itemize}
	\item[(i)] A subspace $\mathcal{Y}$ is strongly anti-coproximinal in $C_0(K, \mathbb{X})$.
	\item[(ii)] for any nonempty open set $U \subset K$ and for any $x^* \in Ext(B_{\mathbb{X}^*})$ there exists an $f \in \mathcal{Y}$ such that $M_f \subset U$ and $J(f(k)) =\{x^*\}~ \forall k \in M_f.$  
	\item[(iii)] for any nonempty open set $U \subset K$ and for any maximal face $F$ of $B_{\mathbb{X}}$ there exists an $f \in \mathcal{Y}$ such that $M_f \subset U$ and $f(k) \in int(F)~ \forall k \in M_f.$  
\end{itemize}
\end{cor}

\begin{proof}
(i) $\iff$ (ii) follows from Theorem \ref{theorem:isolated} directly. \\
(ii) $\iff$ (iii) follows from the fact that for any $x^* \in Ext(B_{\mathbb{X}^*})$ there is an one one corresponds with a maximal face $F$ of $B_{\mathbb{X}}, $ where the connection is $x \in int(F) \iff J(x)=\{x^*\}.$  
\end{proof}

We end this section by providing an example of strongly anti-coproximinal subspaces in support of Theorem \ref{theorem:isolated}.

\begin{example}
Let $K=\{1, 2\}$ and $\mathbb{X}=\ell_1^3.$ In other words, $C(K,\mathbb{X})=  \ell_1^3 \oplus_{\infty} \ell_1^3.$ Consider the following elements of $\ell_1^3.$
\begin{align*}
	u_1 &= (1,1,1),     & u_2 &= (1,-1,-1), & u_3 &= (1,-1,1),   & u_4 &= (1,1,-1), \\
	u_5 &= (2,0,0),           & u_6 &= (-2,0,0),   & u_7 &= (0, 0, 0),    & u_8 &= (0,2,0), \\
	v_1 &= (1,1,0),     & v_2 &= v_3 = (0,0,1), & v_4 &= (-1,1,0), & v_5 &= (1,1,1), \\
	v_6 &= (1,-1,-1),     & v_7 &= (1,-1,1),  & v_8 &= (1,1,-1).
\end{align*}
Let $\mathcal{Y}= span \{ (u_i, v_i): 1 \leq i \leq 8\} \subset C(K, \mathbb{X}).$ It is straightforward to verify that $dim~\mathcal{Y}=5.$
Define the functionals $x_1^*, x_2^*, x_3^*, x_4^* \in (\ell_1^3)^*$ via the canonical isometric isomorphism $\psi$ between $(\ell_1^3)^*$ and $\ell_\infty^3$ as:
\[
\psi(x_1^*) = (1,1,1), \quad \psi(x_2^*) = (1,-1,-1), \quad \psi(x_3^*) = (1,-1,1), \quad \psi(x_4^*) = (1,1,-1).
\]
Clearly, $Ext(B_{(\ell_1^3)^*})= \{ \pm x_1^*, \pm x_2^*, \pm x_3^*, \pm x_4^*\}.$ Then it is easy to observe that    for $x_i^* \in  Ext(B_{(\ell_1^3)^*}$, the corresponding pair $(u_i, v_i)$,  satisfies the property that $\|u_i\| > \|v_i\|$ and $J(u_i)=\{x_i^*\},$ where $1\leq i \leq 4.$ Similarly, for the second coordinate, for any $x_i^* \in  Ext(B_{(\ell_1^3)^*}$, the corresponding pair $(u_{i+4}, v_{i+4})$,  satisfies the property that $\|v_{i+4}\| > \|u_{i+4}\|$ and $J(v_{i+4})=\{x_i^*\},$ where $1\leq i\leq 4.$ Therefore, applying Theorem \ref{theorem:isolated}, we get that $\mathcal{Y}$ is a strongly anti-coproximinal subspace. It should be noted that $\ell_1^3$ does not have any strongly anti-coproximinal subspace \cite[Th. 2.34]{SSGP3}, however the space $\ell_1^3 \oplus_{\infty} \ell_1^3$ does possess such subspaces.  \\
\end{example}

\section{Anti-coproximinality in the space of bounded linear operators}

In this section, we study the anti-coproximinality  (strongly anti-coproximinality)  in the space of all bounded linear operators defined between Banach spaces. Recall that  an operator $T \in \mathbb{L}(\mathbb{X}, \mathbb{Y})$ is said to be an absolutely strongly exposing operator \cite{Jung} if there exists $x_0 \in B_{\mathbb{X}}$ such that whenever a sequence $\{x_n\}$ in $B_{\mathbb{X}}$ satisfies $\|Tx_n\| \to \|T\|,$ then there exists a sequence $\{\theta_n\}\subset \mathbb{T}$ such that $\theta_n x_n \to x_0.$  It is easy to observe that for such an operator $T,$ $M_T = \{\mu x_0: \mu \in \mathbb{T}\},$ for some $x_0 \in S_{\mathbb{X}}.$ The set of absolutely strongly exposing operators of $\mathbb{L}(\mathbb{X}, \mathbb{Y})$ is denoted by $ASE(\mathbb{X}, \mathbb{Y})$. Let us recall the characterization of approximate orthogonality of a $T\in ASE(\mathbb{X}, \mathbb{Y})$ which will be needed in this section. 
\begin{theorem}\label{BSP}\cite[Th. 3.37]{SGSP}
Let $\mathbb{X}$, $\mathbb{Y}$ be two Banach spaces and let $\epsilon \in [0,1).$ Suppose that  $T \in ASE(\mathbb{X}, \mathbb{Y})$ with $M_T= \{ \mu x_0: \mu \in \mathbb{T}\}.$  Then for any $A \in S_{\mathbb{L}(\mathbb{X}, \mathbb{Y})}$, the following are equivalent:
\begin{itemize}
	\item[(i)] $T \perp_B^{\epsilon} A.$
	\item[(ii)] $ \Omega \cap \mathcal{D}(\epsilon)\neq \emptyset,$ where $\Omega = \{y^*(Ax_0): \, y^*\in J(Tx_0) \}.$
	%	\item[(iii)] There exists $x \in M_T$ such that $Tx \perp_B^{\epsilon} Ax.$
\end{itemize} 
Moreover, $T \perp_B A$ if and only if $Tx_0 \perp_B Ax_0.$
\end{theorem}

In the following result, we explore how anticoproximinal subspaces of $\mathbb{Y}$ are related to the space of bounded linear operators $\mathbb{L}(\mathbb{X}, \mathbb{Y})$.

\begin{theorem}\label{operator:anti}
Let $\mathbb{X}$ be a Banach space such that $st \mymathhyphen Exp(B_{\mathbb{X}})$ separates $\mathbb{X}^*$  and $\mathbb{Y}$ be any Banach space.	If $\mathbb{Z}$ is anti-coproximinal in $\mathbb{Y}$, then any subspace $\mathcal{W}$ of $\mathbb{L}(\mathbb{X}, \mathbb{Y})$ containing $\mathcal{F}(\mathbb{X}, \mathbb{Z})$ is anti-coproximinal in $\mathbb{L}(\mathbb{X}, \mathbb{Y}).$
\end{theorem}

\begin{proof}
Let us assume $\mathcal{W} \perp_B A,$ for some $A \in \mathbb{L}(\mathbb{X}, \mathbb{Y}).$ To prove our theorem we show that $A=0.$ Take  $ \widetilde{x} \in st \mymathhyphen Exp(B_{\mathbb{X}})$ and  let $f_{\widetilde{x}} \in B_{\mathbb{X}^*}$ strongly exposes the point $\widetilde{x}.$ Let $z \in S_{\mathbb{Z}}.$  Define $T_z : \mathbb{X} \to \mathbb{Y}$ be such that $ T_z(x) = f_{\widetilde{x}} (x) z ~ \forall x \in \mathbb{X}.$ Clearly, $T_z \in  \mathcal{F}(\mathbb{X}, \mathbb{Y}) \subset \mathcal{W}$ and $\|T_z\|=1.$  Observe that $M_{T_z} =  \{ \mu \widetilde{x}: |\mu |=1\}.$ Let $\{x_n\}\subset S_\mathbb{X}$ be such that $\|T_z (x_n)\| \longrightarrow \|T_z\|.$ Therefore, $\lim \|f_{\widetilde{x}}(x_{n}) z\|= \|T_z\|.$ This implies $\lim |f_{\widetilde{x}} (x_{n})|=1.$ Therefore, there exists a sequence $\{\theta_n\} \subset \mathbb{T}$ such that $\theta_n f_{\widetilde{x}}(x_n) \longrightarrow 1,$ which implies that $f_{\widetilde{x}}(\theta_nx_n) \longrightarrow 1= f_{\widetilde{x}}(\widetilde{x}).$  Since $\widetilde{x}$ is strongly exposed point of $B_{\mathbb{X}}$, we have $ \theta_n x_{n} \longrightarrow \widetilde{x}.$ Therefore $T_z \in ASE(\mathbb{X}, \mathbb{Z}).$ As for any $z \in S_{\mathbb{Z}},$ $T_z \perp_B A,$ and so following Theorem \ref{BSP},  $$ T_z (\widetilde{x}) \perp_B A(\widetilde{x}) \implies z \perp_B A(\widetilde{x}) \implies \mathbb{Z} \perp_B A(\widetilde{x}).$$ As $\mathbb{Z}$ is anti-coproximinal in $\mathbb{X}$, we get $A(\widetilde{x})=0.$ As $st \mymathhyphen Exp(B_{\mathbb{X}})$ separates $\mathbb{X}^*$ and for any $x \in st \mymathhyphen Exp(B_{\mathbb{X}})$, $A(x)=0,$ would imply $A=0.$ Therefore, $\mathcal{W}$ is anti-coproximinal in $\mathbb{L}(\mathbb{X}, \mathbb{Y}).$
\end{proof}

In the above theorem, if we consider an additional  assumption on the subspace of $\mathbb{L}(\mathbb{X}, \mathbb{Y})$ then we obtain a characterization of anti-coproximinality of those subspaces.
\begin{theorem}\label{ASE; anti}
Let $\mathbb{X}, \mathbb{Y}$ be a Banach space such that $st \mymathhyphen Exp(B_{\mathbb{X}})$ separates $\mathbb{X}^*$. Suppose that  $\mathcal{W} \subset \mathbb{L}(\mathbb{X}, \mathbb{Y})$ is a subspace  containing  $\mathcal{F}(\mathbb{X}, \mathbb{Z})$ such that $ASE(\mathbb{X}, \mathbb{Z}) \cap \mathcal{W}$ is dense in $\mathcal{W}.$ Then $\mathcal{W}$
is  anti-coproximinal in $\mathbb{L}(\mathbb{X}, \mathbb{Y})$ if and only if  $\mathbb{Z}$ is anti-coproximinal in $\mathbb{Y}$.
\end{theorem}

\begin{proof}
Since the sufficient part follows from Theorem \ref{operator:anti}, we only need to prove the necessary part. Suppose on the contrary that $\mathbb{Z}$ is not anti-coproximinal in $\mathbb{Y}.$ This implies there exists $y_0 \in \mathbb{Y}$ such that $\mathbb{Z} \perp_B y_0.$ Let $f \in S_{\mathbb{X}^*}.$ Define $A: \mathbb{X} \to \mathbb{Y}$ such that $A(x)= f(x) y_0.$ Let $ T \in ASE(\mathbb{X}, \mathbb{Z}) \cap S_{\mathcal{W}}$ and let $M_T= \{ \mu x_0: \mu \in \mathbb{T}\},$ for some $x_0 \in S_{\mathbb{X}}.$ As $\mathbb{Z} \perp_B y_0,$ we get $Tx_0 \perp_B y_0 \implies Tx_0 \perp_B Ax_0.$  Following Theorem \ref{BSP}, we get $T \perp_B A.$ This implies $ASE(\mathbb{X}, \mathbb{Z})  \perp_B A.$ Since $ASE(\mathbb{X}, \mathbb{Z}) \cap \mathcal{W}$ is dense in $\mathcal{W},$ one can  observe that $\mathcal{W} \perp_B A.$ This contradicts $\mathcal{W}$ is anti-coproximinal in $\mathbb{L}(\mathbb{X}, \mathbb{Y}).$
\end{proof}

Next we investigate the relation between strongly anti-coproximinal subspaces of $\mathbb{Y}$ and $\mathbb{L}(\mathbb{X}, \mathbb{Y}).$
\begin{theorem}
Let $\mathbb{X}$ be a Banach space such that $B_{\mathbb{X}}=\overline{co(st\mymathhyphen Exp(B_{\mathbb{X}}))}$  and $\mathbb{Y}$ be any Banach space.  If $\mathbb{Z}$ be strongly anti-coproximinal in $\mathbb{Y}$ then any subspace $\mathcal{W}$ of $\mathbb{L}(\mathbb{X}, \mathbb{Y})$ containing  $\mathcal{F}(\mathbb{X}, \mathbb{Z})$ is strongly  anti-coproximinal in $\mathbb{L}(\mathbb{X}, \mathbb{Y}).$
\end{theorem}

\begin{proof}
Let $\mathcal{W} \perp_B^\epsilon A,$ for some $\epsilon \in [0,1).$ Take  $ \widetilde{x} \in st \mymathhyphen Exp(B_{\mathbb{X}})$ and  let $f_{\widetilde{x}} \in B_{\mathbb{X}^*}$ strongly exposes the point $\widetilde{x}.$ Let $z \in S_{\mathbb{Z}}.$  Define $T_z : \mathbb{X} \to \mathbb{Y}$ be such that $ T_z(x) = f_{\widetilde{x}} (x) z ~ \forall x \in \mathbb{X}.$ Clearly, $T_z \in  \mathcal{F}(\mathbb{X}, \mathbb{Y}) \subset \mathcal{W}$ and $\|T_z\|=1.$   Following similar argument as in Theorem \ref{operator:anti}, $T_z \in ASE(\mathbb{X}, \mathbb{Z})$ and $M_{T_z}= \{\mu \widetilde{x}: |\mu|=1\}.$ As for any $z \in S_{\mathbb{Z}},$ $T_z \perp_B^\epsilon A,$ following Theorem \ref{BSP},  $$ T_z (\widetilde{x}) \perp_B^\epsilon A(\widetilde{x}) \implies z \perp_B^\epsilon A(\widetilde{x}) \implies \mathbb{Z} \perp_B^\epsilon A(\widetilde{x}).$$ 
As $\mathbb{Z}$ is strongly anti-coproximinal in $\mathbb{X}$, we get $A(\widetilde{x})=0.$ So, for any $x \in co(st \mymathhyphen Exp(B_{\mathbb{X}}))$ such that $ A(x)=0.$  
Let $x \in B_{\mathbb{X}},$ then there exists $\{x_n\} \subset co(st \mymathhyphen Exp(B_{\mathbb{X}}))$ such that $x_n \to x.$  As $ A(x_n)=0 ~ \forall n \in \mathbb{N}, $  we get $A(x)=0.$ This implies $A=0.$
Therefore, $\mathcal{W}$ is strongly anti-coproximinal in $\mathbb{L}(\mathbb{X}, \mathbb{Y}).$
\end{proof}

\begin{theorem}\label{theorem:ASE}
Let $\mathbb{X}$ be Banach space with  $B_{\mathbb{X}}=\overline{co(st\mymathhyphen Exp(B_{\mathbb{X}}))}$ and  $\mathbb{Y}$ be any Banach space.	Suppose that  $\mathcal{W} \subset \mathbb{L}(\mathbb{X}, \mathbb{Y})$ is a closed   subspace  containing  $\mathcal{F}(\mathbb{X}, \mathbb{Z})$ such that $ASE(\mathbb{X}, \mathbb{Z}) \cap \mathcal{W}$ is dense in $\mathcal{W}.$ Then $\mathcal{W}$ is strongly anti-coproximinal in $\mathbb{L}(\mathbb{X}, \mathbb{Y})$ if and only if  $\mathbb{Z}$ is strongly anti-coproximinal in $\mathbb{Y}$.

\end{theorem}

\begin{proof}
We only need to prove the necessary part. Suppose on the contrary that $\mathbb{Z}$ is not strongly anti-coproximinal in $\mathbb{Y}.$ This implies there exists $y_0 \in \mathbb{Y}$ and $\epsilon \in [0,1)$ such that $\mathbb{Z} \perp_B^\epsilon y_0.$ Let $f \in S_{\mathbb{X}^*}.$ Define $A: \mathbb{X} \to \mathbb{Y}$ such that $A(x)= f(x) y_0.$ Let $ T \in ASE(\mathbb{X}, \mathbb{Z}) \cap S_{\mathcal{W}}$ and let $M_T= \{ \mu x_0: \mu \in \mathbb{T}\},$ for some $x_0 \in S_{\mathbb{X}}.$ As $\mathbb{Z} \perp_B^\epsilon y_0,$ we get $Tx_0 \perp_B^\epsilon y_0 \implies Tx_0 \perp_B^\epsilon Ax_0.$  Following Theorem \ref{BSP}, we get $T \perp_B^\epsilon A.$ This implies $ASE(\mathbb{X}, \mathbb{Z}) \cap S_{\mathcal{W}} \perp_B^\epsilon A.$ Since $ASE(\mathbb{X}, \mathbb{Z}) \cap \mathcal{W}$ is dense in $\mathcal{W},$ we get $\mathcal{W} \perp_B^\epsilon A.$ This contradicts $\mathcal{W}$ is strongly  anti-coproximinal in $\mathbb{L}(\mathbb{X}, \mathbb{Y}).$
\end{proof}

As $c_0$ is strongly anti-coproximinal in $\ell_{\infty}$ \cite[Cor. 3.18]{SGSP}, applying Theorem \ref{theorem:ASE} we obtain the following result.

\begin{cor}
Let $\mathbb{X}$ be a Banach space such that $ASE(\mathbb{X}, c_0)$ is dense in $\mathbb{L}(\mathbb{X}, c_0)$ and  $B_{\mathbb{X}}=\overline{co(st\mymathhyphen Exp(B_{\mathbb{X}}))}$. Then $ \mathbb{L}(\mathbb{X}, c_0)$ is strongly anti-coproximinal in $\mathbb{L}(\mathbb{X}, \ell_{\infty}).$ In particular, $\mathbb{L}(\ell_p, c_0)$ is strongly anti-coproximinal in $\mathbb{L}(\ell_p, \ell_{\infty}),$ where $1 < p < \infty.$
\end{cor}

In \cite{Jung}, several conditions are provided under which the set $ASE(\mathbb{X}, \mathbb{Y})$ is dense in a closed subspace of $\mathbb{L}(\mathbb{X}, \mathbb{Y}).$ Whenever $\mathbb{X}$ satisfies Radon-Nikod\'{y}m Property, it follows from  \cite[Th. 5]{B2} that for any given Banach space $\mathbb{Y},$ $ASE(\mathbb{X}, \mathbb{Y})$ is dense in $\mathbb{L}(\mathbb{X}, \mathbb{Y})$. From \cite[p. 121]{DP}, we note that a Banach space $\mathbb{X}$ satisfies the Radon-Nikodym Property if and only if every bounded subset of $\mathbb{X}$ is dentable. Also, in \cite[Th. 9]{P} Phelps proved that  every bounded subset of $\mathbb{X}$ is dentable if and only if every bounded closed convex subset of $\mathbb{X}$ is the closed convex hull of its strongly exposed points. Hence we obtain the following corollary of Theorem \ref{theorem:ASE}.

\begin{cor}
Let $\mathbb{X}$ has Radon-Nikod\'{y}m Property and $\mathbb{Y}$ be any Banach space. Then $\mathbb{L}(\mathbb{X}, \mathbb{Z})$  is (strongly) anti-coproximinal in $\mathbb{L}(\mathbb{X}, \mathbb{Y})$ if and only if $\mathbb{Z}$ is (strongly) anti-coproximinal in $\mathbb{Y}.$ 
\end{cor}

\noindent Another important condition for set $ASE(\mathbb{X}, \mathbb{Y})$ is dense in a closed subspace of $\mathbb{L}(\mathbb{X}, \mathbb{Y})$  is that  $SE(B_\mathbb{X})$ is dense in $S_{\mathbb{X}^*}$ and $\mathbb{Y}$ has $quasi \mymathhyphen \beta$ property \cite{Jung}. A Banach space $\mathbb{Y}$  is said to have property $quasi\mymathhyphen \beta$ \cite{AAP} if there exist
$A = \{y^*_{\lambda} : \lambda \in \Lambda \} \subset S_{\mathbb{Y}^*},$ a mapping $\sigma : A \to S_\mathbb{Y}$, and a function $\rho: A \to \mathbb{R}$ satisfying:
\begin{itemize}
\item [(i)] $y^*_\lambda (\sigma(y^*_\lambda))=1,$ for every $\lambda \in \Lambda,$ 
\item[(ii)] $|z^*(\sigma(y^*))| \leq \rho(y^*) < 1$ whenever $y^*, z^* \in A$ with $y^*= z^*,$
\item[(iii)] for every $e^* \in Ext(B_{\mathbb{Y}^*}),$ there exists a subset $A_{e^*} \subset A$ and $t \in \mathbb{C}$ with $|t| = 1$ such
that $te^* \in \overline{A_{e}^*}^{w^*}$
and $\sup\{\rho(y^*): y^* \in  A_{{e}^*} \} < 1.$
\end{itemize}
Regarding $quasi \mymathhyphen \beta$ property, we recall the following result.
\begin{lemma}\cite[Th. 3.1]{Jung}\label{Jung}
Let $\mathbb{X}, \mathbb{Y}$ be Banach spaces. Suppose that $SE(B_\mathbb{X})$ is dense in $S_{\mathbb{X}^*}$ and that $\mathbb{Y}$
has $quasi \mymathhyphen \beta$ property.  Then, for every closed subspace $\mathcal{W}$ of $\mathbb{L}(\mathbb{X}, \mathbb{Y} )$ containing all
rank one operators, $ASE(\mathbb{X}, \mathbb{Y} )\cap \mathcal{W}$  is dense in $\mathcal{W}$.
\end{lemma}

Now applying Theorem \ref{theorem:ASE} along with Lemma \ref{Jung}, the following result is immediate.

\begin{cor}
Let $\mathbb{X}, \mathbb{Y}$ be Banach spaces such that $SE(B_\mathbb{X})$ is dense in $S_{\mathbb{X}^*}$ and $B_{\mathbb{X}}=\overline{co(st\mymathhyphen Exp(B_{\mathbb{X}}))}.$  Suppose that $\mathbb{Z}$ is a subspace of $\mathbb{Y}$ such that $\mathbb{Z}$ 	has  $quasi \mymathhyphen \beta$ property. Then  any closed subspace $\mathcal{W}\subset \mathbb{L}(\mathbb{X}, \mathbb{Z})$ containing $\mathcal{F}(\mathbb{\mathbb{X}, \mathbb{Z}})$ is  strongly anti-coproximinal in $\mathbb{L}(\mathbb{X}, \mathbb{Y})$ if and only if  $\mathbb{Z}$ is strongly anti-coproximinal in $\mathbb{Y}$.
\end{cor}

\vspace{0.3cm}

We conclude this section  with the a complete description of strongly anti-coproximinal subspaces in the space $\mathbb{L}(\ell_1^n, \mathbb{X})$ and $ \mathbb{K}(\mathbb{X}, C(K))$. It is worth noting that these results follow as a corollary of Theorem \ref{theorem:C(K,X)} and hence the proofs are omitted. Before that we recall these well-known connections from (cf. \cite[Prop. 3.22]{PSS}).
\begin{itemize}
\item The space  $\mathbb{L}(\ell_{1}^n, \mathbb{X})$ is isometrically isomorphic to $\ell_{\infty}^n(\mathbb{X}),$ for any Banach space $\mathbb{X}.$ 

\item  The space  $\mathbb{K}( \mathbb{X}, C(K))$ is isometrically isomorphic to $C(K, \mathbb{X}^*),$ where $K$ is a compact Hausdorff space.

\item  The space $\mathbb{L}( \mathbb{X}, \ell_{\infty}^n)$ is isometrically isomorphic to $\ell_{\infty}^n(\mathbb{X}^*),$
for any Banach space $\mathbb{X}.$ 			
\end{itemize}

Considering the above three connections we mention the following results.
\begin{prop}
Let $\mathbb{X}$ be a Banach space such that $B_{\mathbb{X}^{*}}= \overline{co(w^* \mymathhyphen st \mymathhyphen Exp(B_{\mathbb{X}^{*}}))}^{w^*} .$  Then a subspace $\mathcal{Y}$ is strongly anti-coproximinal in $\mathbb{L}(\ell_1^n, \mathbb{X})$ if and only if for any $k \in \{1, 2, \ldots, n\}$ and for any nonempty weak*-open set  $ V \subset \mathbb{X}^*$ containing a weak*-strongly exposed point of $B_{\mathbb{X}^*}$,  there exists $T \in \mathcal{Y}$ such that 
\begin{itemize}
	\item[(i)]  $M_{T} \cap Ext(B_{\ell_1^n}) = \{ \pm e_k\}$
	\item[(ii)] $J(T(e_k)) \subset V$ or $J(T(-e_k)) \subset V$, 
\end{itemize}
where $e_k =(0, 0, \ldots, \underset{k \mymathhyphen th}{1}, 0, \ldots, 0) \in \ell_{1}^n.$
\end{prop}

\begin{prop}
Let $K$ be a compact Hausdorff space. Let $\mathbb{X}$ be a Banach space such that $B_{\mathbb{X}^{**}}= \overline{co(w^* \mymathhyphen st \mymathhyphen Exp(B_{\mathbb{X}^{**}}))}^{w^*} .$ Then a subspace $\mathcal{Y}$ is strongly anti-coproximinal in $\mathbb{K}(\mathbb{X}, C(K))$ if and only if for any nonempty open set $U \subset K$ and for any nonempty weak*-open set  $ V \subset \mathbb{X}^*$ containing a weak*-strongly exposed point of $B_{\mathbb{X}^*}$,  there exists $T \in \mathcal{Y}$ such that 
\begin{itemize}
	\item[(i)]  $M_{T^*} \cap Ext(B_{C(K)^*})\subset \{ \delta_k: k \in U\}$
	
	\item[(ii)] $J(T^*(\delta_k)) \subset V ~\forall \delta_k \in M_{T^*} \cap Ext(B_{C(K)^*}) $ or $J(T^*(-\delta_k)) \subset V ~\forall \delta_k \in M_{T^*} \cap Ext(B_{C(K)^*}) ,$
\end{itemize}
where $\delta_k : C(K) \to \mathbb{K}$ defined as $\delta_k(f)=f(k) ~\forall f \in C(K).$
\end{prop}

In the next proposition we formulate the above result in the space $\mathbb{L}(\mathbb{X}, \ell_{\infty}^n).$ As $(\ell_{\infty}^n)^*=\ell_{1}^n,$ for an operator $T \in \mathbb{L}(\mathbb{X}, \ell_{\infty}^n),$ we consider $T^* \in \mathbb{L}(\ell_{1}^n, \mathbb{X}^*).$

\begin{prop}
Let $\mathbb{X}$ be a Banach space such that $B_{\mathbb{X}^{**}}= \overline{co(w^* \mymathhyphen st \mymathhyphen Exp(B_{\mathbb{X}^{**}}))}^{w^*}.$   Then a subspace $\mathcal{Y}$ is strongly anti-coproximinal in $\mathbb{L}( \mathbb{X}, \ell_\infty^n)$ if and only if for any $k \in \{1, 2, \ldots, n\}$ and for any nonempty weak*-open set  $ V \subset \mathbb{X}^*$ containing a weak*-strongly exposed point of $B_{\mathbb{X}^*}$,,  there exists $T \in \mathcal{Y}$ such that

\begin{itemize}
	\item[(i)]  $M_{T^*} \cap Ext(B_{\ell_1^n}) = \{ \pm e_k\}$
	
	\item[(ii)]   $J(T^*(e_k)) \subset V$ or $J(T^*(-e_k)) \subset V$,
\end{itemize}
where $e_k =(0, 0, \ldots, \underset{k \mymathhyphen th}{1}, 0, \ldots, 0) \in \ell_{1}^n.$
\end{prop}

\section{Anti-coproximinality in general Banach space}

In this section, our objective is to study anti-coproximinal and strongly anti-coproximinal subspaces in the setting of general Banach spaces. To this end, we introduce the concept of selection maps, which serves as a key tool in providing a complete characterization of anti-coproximinal, strongly anti-coproximinal, coproximinal and co-Chebyshev subspaces. We begin with the following definition.

\begin{definition} Let $\mathbb{Y}$ be a subspace of a Banach space $\mathbb{X}.$

\begin{enumerate}
	\item  We say $\psi$ is a selection map of $\mathbb{Y},$ which is defined as $\psi : \mathbb{Y} \to \mathbb{X}^*,$ where $\psi(y)=f_y,$ for some $f_y \in J(y)$ and $f_{\mu y}= \overline{\mu} f_y,$ for every $\mu \in \mathbb{C}, y \in \mathbb{Y}.$
	\item We define $\Lambda_\mathbb{Y}$ as the collection of all the selection maps of $\mathbb{Y},$ i.e., $$\Lambda_\mathbb{Y} = \{\psi: \mathbb{Y} \to \mathbb{X}^* :\, \psi\, \text{is a selection map of $\mathbb{Y}$}\}.$$
	\item Suppose that $\eta \in \Lambda_\mathbb{Y}.$ Then $\eta$ is said to be a minimal selection map of $\mathbb{Y}$ if $Im~\eta \subset Im~\psi,$ for all $\psi \in \Lambda_\mathbb{Y}.$
	
	% \item We say that $\phi \in \Lambda_\mathbb{Y}$ is a $*$-selection map if 
\end{enumerate}
\end{definition}
We would like to note that in case of smooth Banach space $\mathbb{X},$ if $\mathbb{Y}$ is a subspace of $\mathbb{X}$ then $ \Lambda_\mathbb{Y}$ is singleton. Let us now observe some basic and easy properties of this selection map.

\begin{prop}\label{intersection}
Let $\mathbb{Y}$ be a subspace of $\mathbb{X}.$
\begin{itemize}
	\item [(i)] Suppose that $\phi \in \Lambda_{\mathbb{Y}}.$ Then $\mathbb{Y} \cap (^\perp Im~\phi)= \{0\}.$
	\item[(ii)] Let $dim ~\mathbb{Y}=n.$ Then for any $\phi \in \Lambda_\mathbb{Y},$ $dim~span(Im~\phi) \geq n. $
\end{itemize}
\end{prop}

\begin{proof}
(i) If possible let $(\neq 0)y\in \mathbb{Y}\cap { ^\perp Im~\phi}.$ Then there exists $f \in Im~\phi$ such that $f(y)=\|y\|\neq 0.$ But $y\in {^\perp Im~\phi }$ implies that $f (y)=0$, which is a contradiction.\\

(ii) Suppose on the contrary that there exists $\phi \in \Lambda_\mathbb{Y}$ such that $dim ~span(Im~\phi) =p < n.$ Let $S= \{f_1, f_2, \ldots, $ $ f_p\} \in Im~\phi$ be a basis of $span(Im~\phi).$ First we  claim that $codim(^\perp S) < n.$ Let  $x_i \in \mathbb{X} \setminus \ker f_i,$ for any $1 \leq i \leq p$ such that $\{ x_1, x_2, \ldots, x_p\}$ is linearly independent. Let $x \in \mathbb{X}$. It is immediate that $x$ can be written as $x= \sum_{i=1}^p \alpha_i x_i + h,$ for some $\alpha_i \in \mathbb{K}$ and $h\in \cap_{i=1}^p\ker f_i.$ As $ ^\perp S = \cap_{i=1}^p\ker f_i$, $codim (^{\perp}S) \leq p <n.$ Clearly, $\mathbb{Y} \cap {^\perp S} \neq \{0\}.$ Suppose that $w (\neq 0) \in \mathbb{Y} \cap {^\perp S}.$ Let $\phi(w)= f \in J(w).$ As $ w \in \mathbb{Y},$ $f \in Im~\phi \subset span~S.$ But $w \in {^\perp S},$ this implies $f(w)=0.$ This is a contradiction, which completes the proof.
\end{proof}

\begin{prop}\label{prop:minimal}
Let $\mathbb{Y}$ be a subspace of a Banach space $\mathbb{X}.$ Suppose that $\phi \in \Lambda_{\mathbb{Y}}$ is a minimal selection map of $\mathbb{Y}.$ Then   for any $ f \in Im~\phi,$ there exists $y \in \mathbb{Y}$ such that $Im~ \phi \cap J(y) = \{f\}.$ 
\end{prop}

\begin{proof}
Suppose on the contrary assume that $f_0 \in Im ~\phi,$ such that for any $y \in \mathbb{Y},$ $Im~ \phi \cap (J(y) \setminus\{f_0\} ) \neq \emptyset.$  This implies that for any $y \in \mathbb{Y},$ there exists $f \in Im~ \phi \setminus\{f_0\}$ such that $f \in J(y).$
Consider $\eta: \mathbb{Y} \to \mathbb{X}^*$ such that $\eta (y)= f,$ for some $f \in Im~\eta \setminus \{f_0\}$ and $\eta (\mu y)= \overline{\mu} \eta(y), ~\forall \mu \in \mathbb{C}.$
Therefore, $\eta \in \Lambda_{\mathbb{Y}}$  and $ Im~\eta \subseteqq Im~\phi \setminus \{f_0\},$ which contradicts the minimality of $\phi.$ 
\end{proof}

We next present a  characterization of anti-coproximinal subspaces in a general Banach space, in terms of this newly introduced selection maps.

\begin{theorem}\label{anti-coproximinal}
Let $\mathbb{Y}$ be a subspace of a Banach space $\mathbb{X}.$ Then $\mathbb{Y}$ is anti-coproximinal in $\mathbb{X}$ if and only if $\overline{span (Im~\phi)}^{w^*} =\mathbb{X}^*,~ \forall \phi \in \Lambda_\mathbb{Y}.$
\end{theorem}

\begin{proof}
First we prove the sufficient part. Suppose on the contrary that $\mathbb{Y}$ is not anti-coproximinal in $\mathbb{X}.$ This implies there exists $x \in \mathbb{X} \setminus \mathbb{Y}$ such that $\mathbb{Y} \perp_B x.$ Therefore, following Lemma \ref{James}  for every $y \in \mathbb{Y}$  there exists $ f_y \in J(y)$ such that $x \in \ker f_y.$ Let us define $ \phi: \mathbb{Y} \to \mathbb{X}^*$ as $\phi(y) = f_y$ and  $ \phi(\mu y)= \overline{\mu} f_y, ~\forall \mu \in \mathbb{C}.$ Clearly, $Im~\phi= \{f_y: y \in \mathbb{Y},  \phi(\mu y)= \overline{\mu} f_y, ~\forall \mu \in \mathbb{C} \}$ and $\phi \in \Lambda_{\mathbb{Y}}.$ Observe that $^\perp Im~\phi = \cap_{y\in \mathbb{Y}} \ker f_y.$ Clearly, $x \in {^\perp Im~\phi}.$
From \cite[Th. 4.7]{R}, we obtain that $$(^\perp Im~\phi)^\perp = \overline{span(Im~\phi)}^{w^*}.$$ Following \cite[Th. 1.10.16]{M} we get  $$(^\perp Im~\phi)^* = \mathbb{X}^*/\overline{span(Im~\phi)}^{w^*}.$$ As $x(\neq \theta) \in {^\perp Im~\phi},$ we obtain $\overline{span(Im~\phi)}^{w^*} \subsetneq \mathbb{X}^*.$ This  contradiction  proves the sufficient part of the theorem. 

Let us now prove the necessary part. Suppose on the contrary that there exists a $\phi \in \Lambda_{\mathbb{Y}}$ such that  $\overline{span (Im~\phi)}^{w^*} \subsetneq \mathbb{X}^*.$ Suppose that $ Im~\phi = \{ f_y: y \in \mathbb{Y}, f_y \in J(y)\}.$ Since $(^\perp Im~\phi)^* = \mathbb{X}^*/\overline{span(Im~\phi)}^{w^*},$ it follows that $(^\perp Im~\phi)^* \neq \{0\}.$ This implies that $^\perp Im~\phi \neq \{0\}.$ Let us consider $x (\neq 0) \in {^\perp Im~\phi}.$ Since $ {^\perp Im~\phi} = \cap_{y\in \mathbb{Y}} \ker f_y,$ $x \in  \cap_{y\in \mathbb{Y}} \ker f_y.$ Therefore, following Lemma \ref{James} it is easy to observe that $\mathbb{Y} \perp_B x.$ This contradicts 
the fact that $\mathbb{Y}$ is anti-coproximinal. This proves  the theorem.
\end{proof}

For a smooth Banach space $\mathbb{X}$, $ \Lambda_{\mathbb{Y}}$ is singleton for any subspace $\mathbb{Y} $ of $\mathbb{X}.$ Therefore, \cite[Th. 2.6]{SSGP3} is a immediate corollary of the above theorem. Note that in this case $\mathcal{J}_{\mathbb{Y}}= Im~\phi$ and $ \Lambda_\mathbb{Y}=\{\phi\},$ where $\mathcal{J}_{\mathbb{Y}}= \{ f \in S_{\mathbb{X}^*}: f \in J(y), y \in Sm_{\mathbb{X}} \cap S_{\mathbb{Y}}\}.$

\begin{cor}
Let $\mathbb{Y}$ be a subspace of a smooth space $\mathbb{X}.$  Then the followings are equivalent: 
\begin{itemize}
	\item[(i)] $\mathbb{Y}$ is anti-coproximinal in $\mathbb{X}.$
	\item[(ii)] $\overline{span (Im~\phi)}^{w^*} =\mathbb{X}^*,$~ where $\Lambda_\mathbb{Y}= \{ \phi\}.$ 
	\item[(iii)] $\overline{span \{ f \in S_{\mathbb{X}^*}: f(y) = \|y\|\}}^{w^*}= \mathbb{X}^*.$
\end{itemize}
\end{cor}

Whenever $\mathbb{X}$ is finite-dimensional Theorem \ref{anti-coproximinal} can be expressed as follows.

\begin{cor}
Let $\mathbb{Y}$ be a subspace of an $n$-dimensional Banach space $\mathbb{X}.$ Then $\mathbb{Y}$ is anti-coproximinal in $\mathbb{X}$ if and only if $ dim~span (Im~\phi) = n,~ \forall \phi \in \Lambda_\mathbb{Y}.$
\end{cor}

We next present a characterization of strongly anti-coproximinal subspaces of a general Banach space.

\begin{theorem}
Let $\mathbb{X}$ be a Banach space such that $B_{\mathbb{X}^*}= \overline{co(w^* \mymathhyphen st \mymathhyphen Exp(B_{\mathbb{X}^*}))}^{w^*}.$  Then a closed proper subspace $\mathbb{Y}$ is strongly anti-coproximinal in $\mathbb{X}$ if and only if for every $\phi \in \Lambda_{\mathbb{Y}},$ $ w^* \mymathhyphen str \mymathhyphen Exp(B_{\mathbb{X}^*})\subset \overline{Im~\phi}^{w^*}.$
\end{theorem}

\begin{proof}
First we prove the necessary part. Suppose on the contrary that there exists  $\phi \in \Lambda_{\mathbb{Y}}$ such that $ w^* \mymathhyphen st \mymathhyphen Exp(B_{\mathbb{X}^*}) \not\subset \overline{Im~\phi}^{w^*}.$ Let $g \in  w^* \mymathhyphen st \mymathhyphen Exp(B_{\mathbb{X}^*}) \setminus \overline{Im~\phi}^{w^*}.$ As $g \in w^* \mymathhyphen st \mymathhyphen Exp(B_{\mathbb{X}^*}),$ suppose that there exists $w \in S_{\mathbb{X}}$ such that $g_n (w) \longrightarrow g(w) \implies g_n \xrightarrow{\enskip w^* \enskip} g,$ for any sequence $\{g_n\} \subset S_{\mathbb{X}^*}.$   Let $$\epsilon= \sup \{ |f(w)|: f \in Im ~ \phi \}.$$ Observe that $\epsilon < 1.$ Otherwise if $\epsilon=1,$ then there exists $ \{f_n\} \subset  Im ~\phi$ such that $ |f_n(w)| \to 1. $ This implies $ \theta_n f_n(w) \to 1$, for some $\theta_n \in \mathbb{T}.$ Therefore, $$  \psi(w)( \theta_n f_n) \to 1 \implies \theta_n f_n \xrightarrow{\enskip w^* \enskip} g.$$ As $\theta_n f_n \in Im~\phi,$ we get $g \in \overline{Im~\phi}^{w^*},$ which is a contradiction.
We now show that $\mathbb{Y} \perp_B^\epsilon w.$ Let $y \in \mathbb{Y}$ and  $f \in J(y) \cap Im~\phi.$   Clearly, $|f(w)| \leq \epsilon= \epsilon \|w\|.$ Following Lemma \ref{Chem}, $y \perp_B^\epsilon w,$ consequently, $\mathbb{Y} \perp_B^\epsilon w.$ This contradicts that $\mathbb{Y}$ is strongly anti-coproximinal.

Let us now prove the sufficient part. Suppose on the contrary that $\mathbb{Y}$ is not strongly anti-coproximinal. This implies there exists an $\epsilon \in [0,1)$ and $x \in S_{\mathbb{X}}$ such that $\mathbb{Y} \perp_B^\epsilon x.$ Therefore, for any $y \in \mathbb{Y}$ there exists $f_y \in J(y)$ such that $|f_y(x)| \leq \epsilon.$ 
Let us define $ \phi: \mathbb{Y} \to \mathbb{X}^*$ as $\phi(y) = f_y$ and  $ \phi(\mu y)= \overline{\mu} f_y, ~\forall \mu \in \mathbb{C}.$
Clearly, 
\[Im~\phi= \{f_y: y \in \mathbb{Y},  \phi(\mu y)= \overline{\mu} f_y, ~\forall \mu \in \mathbb{C} \}\]
and
$\phi \in \Lambda_{\mathbb{Y}}.$  Let $f \in \overline{Im~\phi}^{w^*}.$ So, there exists $f_{y_n} \in Im~\phi$ such that $f_{y_n} \xrightarrow{\enskip w^* \enskip} f.$ As $| f_{y_n} (x) | \leq \epsilon$ for any $n \in \mathbb{N}$, we get $|f(x)| \leq           \epsilon.$  
From hypothesis $ w^* \mymathhyphen st \mymathhyphen Exp(B_{\mathbb{X}^*})\subset \overline{Im~\phi}^{w^*}.$
So, for any $f' \in  w^* \mymathhyphen st \mymathhyphen Exp(B_{\mathbb{X}^*}),$ $|f'(x)| \leq \epsilon.$
Let $g \in B_{\mathbb{X}^*}.$ As $B_{\mathbb{X}^*}= \overline{co(w^* \mymathhyphen st \mymathhyphen Exp(B_{\mathbb{X}^*}))}^{w^*},$ there exists a net $ \{g_\alpha\} \subset co( w^* \mymathhyphen str \mymathhyphen Exp(B_{\mathbb{X}^*}))$ such that $g_\alpha \overset{w^*}{\longrightarrow} g.$ So, $g_{\alpha}(x) \to g(x).$ Clearly, for any $\alpha,$ $|g_{\alpha} (x)|\leq \epsilon$ and therefore $|g(x)| \leq \epsilon .$ Consequently, $$\|x\|= \sup \{ |g(x)|: g \in B_{\mathbb{X}^*}\} \leq \epsilon < 1=\|x\|. $$ This contradiction completes the proof. 
\end{proof}

For a finite-dimensional Banach space, the above theorem can be expressed as follows.

\begin{cor}
Let $\mathbb{Y}$ be a proper subspace of a finite-dimensional Banach space $\mathbb{X}$. Then $\mathbb{Y}$ is strongly anti-coproximinal if and only if $Exp(B_{\mathbb{X}^*}) \subset \overline{Im~\phi}, ~\forall\phi \in \Lambda_\mathbb{Y}.$
\end{cor}

While we  provide a complete characterization of  the anti-coproximinal and the strongly anti-coproximinal subspaces in general  setting via the notion of selection maps, determing such subspaces from a computational  standpoint remains a formidable challenge. In particular, obtaining the full structure of the set of selection maps for a subspace of  the space $C(K, \mathbb{X})$ and $\mathbb{L}(\mathbb{X}, \mathbb{Y})$ is a difficult task.   Nevertheless, this does not diminish the significance of the notion of slection maps. Utilizing this concept, we can fully characterize the points from which the best coapproximation exists, leading to a complete characterization of coproximinal and co-Chebyshev subspaces. Furthermore, this notion also enables us to characterize when a subspace of a Banach space is isonetrically isomorphic to $\ell_{\infty}^n.$

We now describe the set $dom ~\mathcal{R}_\mathbb{Y}$ in terms of selection map.

\begin{theorem}
Let $\mathbb{Y}$ be a subspace of a Banach space $\mathbb{X}.$ Then
\[dom~\mathcal{R}_\mathbb{Y} = \mathbb{Y} + ( \cup_{\phi \in \Lambda_\mathbb{Y}}{^\perp Im~\phi }).\]
\end{theorem}

\begin{proof}

Let $x \in dom~\mathcal{R}_\mathbb{Y}.$ Then there exists $y_0\in \mathbb{Y}$ such that $y_0 \in \mathcal{R}_\mathbb{Y}.$ This implies that $ \mathbb{Y} \perp_B x-y_0.$ Therefore, $x-y_0 \in {^\perp Im~\phi},$ for some $\phi \in \Lambda_\mathbb{Y}.$ Thus we obtain that $x-y_0 \in \cup_{\phi \in \Lambda_\mathbb{Y}}{^\perp Im~\phi},$ i.e., $dom~\mathcal{R}_\mathbb{Y} \subset  \mathbb{Y} + ( \cup_{\phi \in \Lambda_\mathbb{Y}}{^\perp Im~\phi }).$ 

To show the converse inclusion suppose that $u= y+z,$ where $y\in \mathbb{Y}$ and $z \in \cup_{\phi \in \Lambda_\mathbb{Y}}{^\perp Im~\phi}.$ Then  there exists $\phi \in \Lambda_\mathbb{Y}$ such that $u-y \in {^\perp Im~\phi}.$ This implies that $ \mathbb{Y} \perp_B u-y \implies y \in \mathcal{R}_\mathbb{Y}(u) \implies u \in dom~\mathcal{R}_\mathbb{Y}.$ Therefore,  $\mathbb{Y} +  (\cup_{\phi \in \Lambda_\mathbb{Y}}{^\perp Im~\phi}) \subset dom~\mathcal{R}_\mathbb{Y}.$ This proves the theorem.
\end{proof}

As an immediate corollary we obtain a characterization of coproximinal and co-Chebyshev subspaces.

\begin{cor}\label{corpxoximinal}
Let $\mathbb{Y}$ be a subspace of a Banach space $\mathbb{X}.$ Then the following holds:
\begin{itemize}
	\item[(i)] $\mathbb{Y}$ is coproximinal if and only if  $ \mathbb{Y} + ( \cup_{\phi \in \Lambda_\mathbb{Y}}{^\perp Im~\phi }) = \mathbb{X}.$
	
	\item[(ii)] $\mathbb{Y}$ is co-Chebyshev if and only if $ \mathbb{Y}  \oplus ( \cup_{\phi \in \Lambda_\mathbb{Y}}{^\perp Im~\phi }) = \mathbb{X}.$
\end{itemize}
\end{cor}

\begin{cor}\label{prop:coproximinal}
Let $\mathbb{Y}$ be an $m$-dimensional subspace of a Banach space $\mathbb{X}.$ If there exists $\phi \in \Lambda_\mathbb{Y}$ such that $dim~span(Im~\phi)=m.$ Then $\mathbb{Y}$ is coproximinal in $\mathbb{X}.$ 
\end{cor}

\begin{proof}
As $dim~span(Im~\phi)=m,$ it is clear that $ codim({^\perp Im~\phi})=m.$ From Proposition \ref{intersection}, $\mathbb{Y} \cap {^\perp Im~\phi}= \emptyset.$ Therefore it is immediate that $\mathbb{Y} + {^\perp Im~\phi} =\mathbb{X}$ and therefore, using Corollary \ref{corpxoximinal}, $\mathbb{Y}$ is coproximinal in $\mathbb{X}.$
\end{proof}

\begin{lemma}\label{lemma:facet}
Let $\mathbb{X}$ be a finite-dimensional real polyhedral  Banach space and let $\mathbb{Y}$ be a proper subspace of $\mathbb{X}.$ Suppose that $F$ is a facet of $B_{\mathbb{X}}$ such that $int (F) \cap \mathbb{Y} \neq \emptyset.$ Then $F \cap \mathbb{Y}$ is a facet of $B_{\mathbb{Y}}.$
\end{lemma}

\begin{theorem}\label{general}
Let $\mathbb{Y}$ be an $n$-dimensional subspace of a real Banach space  $\mathbb{X}.$  Suppose that  $\phi \in \Lambda_\mathbb{Y}$ is a minimal selection map of $\mathbb{Y}$  such that $|Im~\phi|=2r.$ 
Then the following holds:
\begin{itemize}
	\item[(i)] $\mathbb{Y}$ is embedded in $\ell_{\infty}^r.$
	\item[(ii)] the number of facets of $B_{\mathbb{Y}}$ is $2r.$
	\item[(iii)] $\mathbb{Y}$ can not be embedded into $\ell_{\infty}^{m},$ for $m < r$.
	\item[(iv)] $\mathbb{Y}$ is isometrically isomorphic to $\ell_{\infty}^n$ if and only if $r=n.$
\end{itemize}  
\end{theorem}

\begin{proof}
(i) Let $Im~\phi= \{ \pm f_1, \pm  f_2, \ldots, \pm f_r\}.$ We define a map $\zeta$ from $\mathbb{Y}$ to $\ell_{\infty}^r$ as follows
\begin{eqnarray*}
	\zeta &:& \mathbb{Y} \longrightarrow \ell_{\infty}^r \\
	\zeta(y) &=& (f_1(y), f_2(y), \ldots, f_r(y)), ~ \forall y \in \mathbb{Y}.
\end{eqnarray*}
Clearly, $\zeta$ is linear. Since $\phi \in \Lambda_{\mathbb{Y}}$, then for any $y \in \mathbb{Y},$ there exists $f \in Im~ \phi$ such that $f \in J(y).$ Therefore,
\begin{eqnarray*}
	\|y\| &=&\max \{ |f_i(y)|: 1 \leq i \leq r\}\\ & =& \|(f_1(y), f_2(y), \ldots, f_r(y))\| \\ &=& \|\zeta(y)\|.
\end{eqnarray*}
So, $\zeta $ is an isometry and $\mathbb{Y}$ can be embedded into $\ell_{\infty}^r.$\\

(ii) Since $\mathbb{Y}$ can be embedded into 	$\ell_{\infty}^r,$ $\mathbb{Y}$ is polyhedral. Let 
\[F_i= \{ \widetilde{x}=(x_1, x_2, \ldots, x_r) \in \ell_{\infty}^r: x_i=1, |x_j| \leq 1~\forall j \neq i \}.\]
Clearly, 
$\pm F_1, \pm F_2, \ldots, \pm F_r$ are the facets of $B_{\ell_{\infty}^r}.$  
Following Proposition \ref{prop:minimal},
for each $i,1 \leq i \leq r,$ there exists $y_i \in \mathbb{Y}$ such that $ Im~\phi \cap J(y_i)=\{f_i\}.$ We already observe that $\zeta,$ (defined in (i)) is a linear isometry from $\mathbb{Y}$ to $\ell_{\infty}^r,$ so $\mathbb{Y}$ is isometrically isomorphic to $\zeta(\mathbb{Y}),$ which is a subspace of $\ell_{\infty}^r.$
As $ Im~\phi \cap J(y_i)=\{f_i\},$ we get $f_i(y_i) = \|y_i\|$ and $|f_j(y_i)| < \|y_i\|, $ for any $j \in \{1, 2, , \ldots, r\} \setminus \{i\}.$ This implies  
$\zeta(\frac{y_i}{\|y_i\|}) \in int(F_i).$  Therefore, from Lemma \ref{lemma:facet}, $F_i \cap \zeta(\mathbb{Y})$ is a facet of $B_{\zeta(\mathbb{Y})},$ for every $i, 1 \leq i \leq r.$ So, the number of facets of $B_{\zeta(\mathbb{Y})}$ is $2r$ and consequently the number of facets of $B_{\mathbb{Y}}$ is $2r.$\\

(iii) If $\mathbb{Y}$ can be embedded into $\ell_{\infty}^m,$ for some $m \in \mathbb{N},$ then the number of facets of $B_{\mathbb{Y}}$ is less than or equal to $2m.$ Then from (ii), the third part is immediate.\\

(iv) The necessary part follows from the fact that from  (ii) we get  $r \geq n$ and  applying Proposition \ref{intersection}(ii) we get $r \leq n.$ For the sufficient part, when $r=n$, the isometry follows from (i) and isomorphism follows from the fact that  $dim \mathbb{Y}=n$. 
\end{proof}

\begin{cor}\label{cor:2n}
Let $\mathbb{Y}$ be an $n$-dimensional subspace of a real Banach space  $\mathbb{X}.$ Suppose that  $\psi$ is a minimal selection map of $\mathbb{Y}$ such that $|Im~\psi|=2r.$ Let us consider the followings conditions:
\begin{itemize}
	\item[(i)] $r=n.$
	\item[(ii)]  $\mathbb{Y}$ is isometrically isomorphic to $\ell_{\infty}^n$.
	\item[(iii)] $\mathbb{Y}$ is coproximinal.
\end{itemize}
Then (i) $\iff $ (ii) $\implies$ (iii).
\end{cor}

\begin{proof}
(i) $\iff$ (ii) follows from theorem \ref{general} and (ii) $\implies$ (iii) follows from Proposition \ref{prop:coproximinal}.
\end{proof}

We end this article by characterizing when a subspace can be polyhedral in a Banach space using the selection maps.

\begin{theorem}\label{polyhedral}
Let $\mathbb{Y}$ be an $n$-dimensional subspace of a real Banach space  $\mathbb{X}.$ Then  $\mathbb{Y}$ is polyhedral if and only if there exists a selection map $\phi\in \Lambda_\mathbb{Y}$  such that $|Im~\phi|$ is finite.
\end{theorem}

\begin{proof}
The sufficient part follows directly from Theorem \ref{general}. To prove the necessary part we assume that $\mathbb{Y}$ is polyhedral and  the number of facets of $B_{\mathbb{Y}}$ is $2r.$ From \cite[Lemma 2.1]{SPBB}, there exists an one-one  corresponding between a facet of $B_{\mathbb{Y}}$ and an extreme point of $B_{\mathbb{Y}^*}.$ Let $Ext(B_{\mathbb{Y}^*})= \{ \pm g_1, \pm  g_2, \ldots, \pm g_r\}.$ It is easy to observe that for any $y \in \mathbb{Y},$ there exists $g \in Ext(B_{\mathbb{Y}^*})$ such that $g(y)=\|y\|.$ Since extreme points in the dual unit ball of a subspace can be extended to the extreme points of dual unit ball of the whole space (see \cite{Rao3}), suppose that $f_i \in \mathbb{X}^*$ is the extension of $g_i,$ for each $i, 1 \leq i \leq r.$ Consider $\psi: \mathbb{Y} \to \mathbb{X}^*$ such that $\psi(y)= f_i, $ whenever $g_i(y) = \|y\|.$ Clearly, $f_i \in J(y).$ Therefore $\psi$ is a selection map of $\mathbb{Y}.$ This completes the theorem.  
\end{proof}

\section*{Declarations}

\begin{itemize}
	\item Funding
	
	NIL
	\item Conflict of interest
	
	The authors have no relevant financial or non-financial interests to disclose.
	
	\item Data availability 
	
	The manuscript has no associated data.
	
	\item Author contribution
	
	All authors contributed to the study. All authors read and approved the final version of the manuscript.
	
\end{itemize}

\end{document}